\documentclass[12pt,oneside,openany]{article}
\usepackage{tikz}
\usepackage{tikz-cd}

\usepackage{amssymb}   
\usepackage{amsthm}
\usepackage{amsmath}
\usepackage{amsfonts}
\usepackage{latexsym}
\usepackage{graphics}
\usepackage{fancyhdr}
\usepackage{epstopdf}
\usepackage{setspace}
\usepackage{url}
\usepackage[top=1in, bottom=1in, left=.75in, right=.75in]{geometry}

\usepackage[vcentermath,enableskew]{youngtab}

\newtheorem{theorem}{Theorem}[section]

\newtheorem{prop}[theorem]{Proposition}

\newtheorem{corollary}[theorem]{Corollary}
\newtheorem{definition}[theorem]{Definition}

\def\R{{\mathbb R}} 
\def\N{{\mathbb N}} 

\title{The Natural Extension for the Triangle Map\\ (a Multi-dimensional Continued Fraction)\\ with \\ An Internal Symmetry from Young Conjugation}
\author{Joe Fox and Thomas Garrity\\
Department of Mathematics\\
Williams College\\
Williamstown, MA 01267\\
tgarrity@williams.edu}

\begin{document}
\maketitle

\begin{abstract} The natural extension of the triangle map  (a type of multi-dimensional continued fraction algorithm) is completely described  in all possible dimensions.  The motivation and inspiration for this natural extension stem from  the triangle map's recent link to the classical study of integer partitions.   Inspired by Young conjugation for integer partitions, we show that the natural extension has an internal symmetry and allows the natural extension to be subdivided into four natural subdomains.  This appears to be new even for the classical case of the natural extension for continued fractions, namely for both the classical Gauss map and the classical  Farey map.

 \end{abstract}

\section{Introduction}

During the COVID lockdown, Bonanno, del Vigna, Isola and the second author realized that  a particular multi-dimensional continued fraction algorithm (the triangle map) could be used to gain insight into classical integer partition theory.  In the year or so after the main part of the lockdown, Wael Baalbaki joined us, resulting in the papers \cite{BBDGI, BG}.  This work critically used the triangle map to give us information about integer partitions.  We of course wondered about turning this around, and using integer partitions to discover something new about the triangle map (which in the one-dimensional case is the classical Gauss map and Farey map for continued fractions).

The current two authors  strongly suspected that these earlier papers  provided a possible method for producing natural extensions for the triangle map.  This is the goal of this paper.  We will see that integer partitions provide a quite clean geometric model for the natural extensions.  Further, the most natural symmetry for integer partitions is Young conjugation (just flip the Young shapes).  While we have no Young shapes, we do have a Young conjugation. We use the Young conjugation to find an internal symmetry in the natural extension, for each dimension.  This seems to be new even for the classical continued fraction natural extension. 

In Section \ref{natural extension}, we recall the basics of natural extensions, with  our goal being to find concrete representations for four different versions of the triangle map; the homogeneous (slow) triangle map, the non-homogeneous (slow), the homogeneous fast triangle map and the non-homogeneous fast triangle map. In Section \ref{Partitions}, we give a quick overview of integer partition theory.  In Section \ref{The homogeneous triangle map} we recall the homogeneous triangle map.  It is in Section \ref{The homogeneous extended triangle map} that we start new material.  It is here that we find the natural extension of the homogeneous triangle map. In Section \ref{affine projection} we find the natural extension for the non-homogeneous triangle map.  In Section \ref{``Graphs'' of Natural Extensions}, we first find the explicit graphs of the natural extension in the classical Farey $m=1$ case  and then use this to get ``graphs'' in higher dimensions.  In Section \ref{Young conjugation: An Internal Symmetry}, we develop the Young conjugate map and find new internal symmetries for the natural extensions.  We view it as  interesting that this symmetry is motivated from integer partition theory.  In Section \ref{fast homogeneous case}, Section \ref{nonhomogeneous fast case}  and  Section \ref{Graphs of the fast natural extensions}, we show analogous results for the fast versions of the algorithms.  In Section \ref{The invariant measure}, we use our work to find the relevant invariant measure for our maps.  Finally, in Section \ref{Siukaev}, we discuss the links of this paper with the recent work of Siukaev \cite{Siukaev}.

\subsection{Acknowledgements} 
We would like to thank Pierre Arnoux, Charles Fougeron, Kevin Jiang, Jacob Lehmann Duke, Lori Pedersen, Tanatswa Manyakara, Cesar Silva and David Siukaev for useful conversations.

\section{Background on natural extension}\label{natural extension}

\subsection{General idea of natural extension}\label{basics of natural extension}
We will be following and using the notation in Section 5.3 in Dajani and Kalle \cite{Dajani-Kalle}.

Let 
$$S:Y\rightarrow Y$$
be a measure preserving transformation on a measure space $(Y, \mathcal{G}, \nu),$ where $\mathcal{G}$ is the $\sigma$-algebra of measurable sets and $\nu$ is the measure. 
Rarely is the map $S$ one-to-one. 

The natural extension is a remarkable associated measure preserving transformation
$$T:X\rightarrow X$$
on a measure preserving space $(X, \mathcal{F}, \mu)$  such that 

\begin{enumerate}
\item the map $T$ is  a one-to-one function
\item there is an onto map 
$$\psi:X\rightarrow Y$$
such that  $\mathcal{F}$ is the $\sigma$-algebra containing all the $\sigma$-algebra $T^{(k)} \psi^{-1} (\mathcal{G})$ for all $k\geq 0.$
\item We have the commutative diagram

$$\begin{matrix} 
X & \xrightarrow{T} & X \\
\psi \downarrow &  & \downarrow \psi \\
Y & \xrightarrow{S} & Y \\
\end{matrix}$$

\item Every other such covering of $Y$ must factor through $X$

\end{enumerate}

It is the last point why the adjective ``natural'' is categorical in nature.
The second point is  linking the dynamics of $Y$ to the dynamics of $X$.

Theorem 5.3.1 (Rohlin) in \cite{Dajani-Kalle} is that the natural extension always exists, which is  where you can find a proof.  This theorem has another key part, namely that dynamical properties of the original many-to-one map $S:Y\rightarrow Y$ are passed on to the new one-to-one map $T:X   \rightarrow X.$  For example, the map $S$ is ergodic if and only if $T$ is ergodic.

It is relatively straightforward to find what $X$ is set-theoretically, to determine the map $T$ and to find the projection map $\psi$.  Simply set  

\begin{eqnarray*}
X&=& \{ (y_0, y_1, y_2, \ldots )\in Y\times Y \times Y \times \cdots : S(y_{n+1}) =y_n \}   \\
T(y_0, y_1, y_2, \ldots ) &=& (y_1, y_2, \ldots ) \\
\psi(y_0, y_1, y_2, \ldots ) &=& y_0.
\end{eqnarray*}

\subsection{Find the geometrical model } \label{m1 case} 

The previous section's construction of the natural extension is quite abstract.  For a specific map, we would like to be able to see if its  natural extension can be made to be somewhat more concrete.  Attempts to do this for various multidimensional continued fraction algorithms is the goal of   Arnoux and Labb\'{e}'s \cite{Arnoux-Labbe}, though this paper did not deal with the triangle map (the map of this paper). To get a more general feel of how the natural extension of a map can be used to understand number theory type questions, see Dajani and Kraaikamp \cite{Dajani-Kraaikamp}.

Arnoux and S. Labb\'{e}'s work was  built on Arnoux and Nogueira's \cite{Arnoux-Nogueira-93}  who found   concrete representation of the natural extension for the  classical Gauss map.  Recall that the Gauss map
$$G:(0,1) \rightarrow (0,1),$$
is defined as
$$G(x) = \frac{1}{x} - \left\lfloor \frac{1}{x} \right\rfloor,$$
where $\lfloor \alpha \rfloor$ is the greatest integer part of the real number $\alpha$.
This is an infinite-to-one map and hence far from being one-to-one. 

Arnoux and Nogueira showed that the domain of the natural extension is 
$$\Sigma = \left\{ (x, t) \in \R^2: 0<x<1, 0< t< \frac{1}{1+x} \right\}
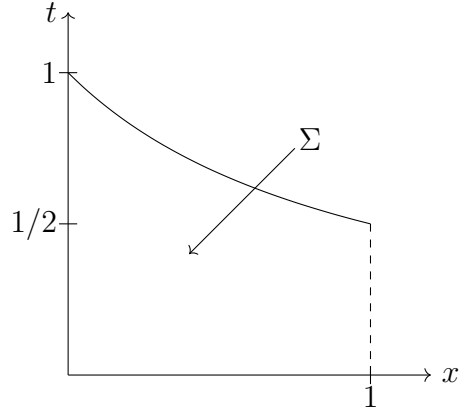
$$

\begin{center}
\begin{tikzpicture}[scale=4]

\draw[->](0,0)--(1.2, 0);
\draw[->](0,0)--(0, 1.2);
\draw[dashed](1,0)--(1, 1/2);

\node[right] at (1.2,0){$x$};
\node[left]at (0,1.2){$t$};

\draw[-](1,.03)--(1, -.03);
\node[below] at (1,0){$1$};

\draw[-](.03, 1)--(-.03, 1);
\node[left] at (0,1){$1$};

\draw[-](.03, 1/2)--(-.03, 1/2);
\node[left] at (0,1/2){$1/2$};

\draw[scale=1, domain=0:1, smooth, variable=\x] plot ({\x}, {1/(1+ \x)});

\draw[->](3/4,3/4)--(.4, .4);
\node[] at (.8,.78){$\Sigma$};

\end{tikzpicture}
\end{center}
with corresponding map
$$\tilde{G}(x,t) = \left( \frac{1}{x} - \left\lfloor \frac{1}{x} \right\rfloor, x-x^2 t \ \right).$$
We will later be reproving this in a more general context.  More importantly, we will see that there is a  natural involution map on $\Sigma$ corresponding to Young conjugation for integer partitions.  This involution is new.

Our inspiration for finding these geometric models will stem from integer partition theory.  Further we already know that the triangle map, in all dimensions, is ergodic  \cite{Garrity-Lehmann Duke}.  Thus we know that the natural extensions exist.  We simply have to find them.

\section{A rapid review of integer partitions: the basics}\label{Partitions}

An integer partition of a number $N$ is given by a sequence of positive integers $\lambda_0>\lambda_1>\cdots > \lambda_m$ and positive integers $k_0, \ldots, k_m$ such that 
$$k_0 \lambda_0 + \ldots + k_m \lambda_m = N.$$
The $\lambda_i$ are called the parts and the $k_j$ are called the multiplicities.  In the integer partition community, the partition is usually written as 
$$(\lambda_0^{k_0}, \ldots, \lambda_m^{k_m}) \vdash N.$$
 It is here traditional to list the seven  partitions of $N=5$:
 $$(5^1), (4^1,1^1), (3^1,2^1),(3^1,1^2), (2^2,1^1), (2^1,1^3),  (1^5)$$
$$(5) \times [1], (4,1) \times [1,1], (3,2)\times [1,1], (3, 1) \times[1,2], (2,1) \times [2,1], (2,1) \times [1,3] , (1) \times [5]$$
corresponding to 
$$5, 4+1, 3+2, 3+1+1, 2+2+1, 2+1+1+1, 1+1+1+1+1.$$

Using the notation 
$$(\overline{\lambda})= (\lambda_0, \ldots, \lambda_m), \; [\overline{k}] =[k_0, \ldots, k_m]$$
we will instead write partitions as:
$$(\overline{\lambda}) \times [ \overline{k}] \vdash N,$$
which can be interpreted as the following dot product
$$ [ \overline{k}]  \cdot (\overline{\lambda})^{\top}= N.$$

The dot product can be interpreted as a symplectic form, a rhetoric that sounds pretentious to people in the integer partition community but not that strange for people in the dynamical system community who are interested in natural extensions.  

The papers  \cite{BBDGI, BG} developed an almost internal symmetry on the space of partitions.  One of the themes in these two papers stemmed from the following.  Starting with  the parts vector $(\overline{\lambda})$ and multiplicity vector $[\overline{k}] $ for  fixed number $m$, suppose we have an $(m+1) \times (m+1)$ invertible matrix $A$ with integer entries that are all non-negative and so that all entries of the inverse matrix $A^{-1}$ are also non-negative integers.  Then we have 
$$N=[k_0, \ldots, k_m] \cdot \left(  \begin{array}{c} \lambda_0 \\ \vdots \\ \lambda_m \end{array} \right)  = [k_0, \ldots, k_m] A^{-1} \cdot A \left(  \begin{array}{c} \lambda_0 \\ \vdots \\ \lambda_m \end{array} \right).  $$
This means that given an integer partition $ (\overline{\lambda}) \times [ \overline{k}] $ of $N$, we can immediately write down the new integer partition 
$$(\overline{\lambda}A^{\top}) \times [ \overline{k}A^{-1}] .$$
Such $A$ are actually hard to find.

The papers  \cite{BBDGI, BG} used the triangle map to create two such $A$ to provide a map from partitions to partitions.  This in turn will allow us to find the natural extension of the triangle map.

\section{The homogeneous triangle map}\label{The homogeneous triangle map}
There are many different multi-dimensional continued fraction algorithms, each with their own properties, each with their own delights.  Here we will be looking at one particular one, the triangle map.  Good sources for general   background on multidimensional continued fraction algorithms are in  Karpenkov \cite{Karpenkov 1} and Schweiger \cite{Sch} .  

The original motivation for the development of the triangle map (the $m=2$ case) was the number-theoretic  Hermite problem \cite{Garrity1, Ass}.   Messaoudi,  Nogueira, and  Schweiger \cite{Mes} 
showed that  this map  is ergodic and discussed why it is dynamically interesting.  
As mentioned in \cite{BBDGI}, ``further dynamical properties were discovered by  Berth\'e, Steiner and Thuswaldner \cite{Berthe-Steiner-Thuswaldner}  and by Fougeron and  Skripchenko \cite{Fougeron-Skripchenko}.   Bonanno, Del Vigna and  Munday \cite {Bonanno- Del Vigna-Munday} and Bonanno and Del Vigna \cite{Bonanno-Del Vigna} recently used the $\R^3$ slow  triangle map to develop a tree structure of rational pairs in the plane.   In a recent preprint, Ito \cite{Ito} showed that the fast  map is self-dual (in section three of that paper).  These papers are all primarily motivated by questions from dynamics.'' There is also the recent paper \cite{Garrity-Lehmann Duke} which proves that the triangle map is ergodic in all dimensions.

Our motivation is that this particular multi-dimensional continued fraction algorithm can be used to find additional structure on integer partitions \cite{BBDGI, BG}.  Further, once we think in terms of integer partitions, we will see that we are led to an additional internal symmetry on the natural extension space, one which will not exist for the natural extensions of other multidimensional continued fractions.  In fact, this will provide an additional symmetry even for the classical natural extensions of the Gauss map.  This is in part what justifies this paper.

But here we just write down the triangle map.  Fix an integer $m>1$.  Set
\begin{eqnarray*}
\triangle &=& \{    (\lambda_0, \ldots, \lambda_m)\in \R^{m+1}: \lambda_0 > \ldots > \lambda_m >0 \} \\
\triangle_0 &=& \{(\lambda_0, \ldots, \lambda_m)\in \triangle: \lambda_0 < \lambda_1+\lambda_m \} \\
\triangle_1 &=& \{(\lambda_0, \ldots, \lambda_m)\in \triangle: \lambda_0 >\lambda_1+\lambda_m \} \\
\triangle_D &=& \{(\lambda_0, \ldots, \lambda_m)\in \triangle: \lambda_0 =\lambda_1+\lambda_m \} \\
\end{eqnarray*}
When we want to emphasize the number of variables, we will write these cones as $$  \triangle(m) ,  \triangle_0 (m), \triangle_1 (m), \triangle_D (m) .$$

Set
$$ \overline{\lambda}=  (\lambda_0, \ldots, \lambda_m).$$

\begin{definition} The (homogeneous) triangle map 
$$T:\triangle_0 \cup \triangle_1 \rightarrow \triangle$$
is  the map
\begin{eqnarray*}
T(\overline{\lambda}) &=& \left\{ \begin{array}{ccc}   T_0(  \overline{\lambda}) & \mbox{if} &  \overline{\lambda}\in \triangle_0 \\
T_1(  \overline{\lambda}) & \mbox{if} &  \overline{\lambda}\in \triangle_1\end{array}   \right.\\
&=& \left\{ \begin{array}{ccc} (\lambda_1, \lambda_2, \ldots, \lambda_m, \lambda_0-\lambda_1)& \mbox{if} &  \overline{\lambda}\in \triangle_0 \\
(\lambda_0-\lambda_m, \lambda_1, \lambda_2, \ldots, \lambda_m) & \mbox{if} &  \overline{\lambda}\in \triangle_1\end{array}   \right.
\end{eqnarray*}
\end{definition}

We can describe both $T_0$ and $T_1$ as $(m+1) \times (m+1)$ matrices:

$$\begin{array}{ccccc}
T\left(\overline{\lambda}^{\top}\right) &=&  T\left( \begin{array}{c} \lambda_0 \\ \lambda_1 \\ \vdots  \\ \lambda_m  \end{array}  \right)&=&  \left\{ \begin{array}{ccc} T_0 \left( \begin{array}{c} \lambda_0 \\ \lambda_1 \\ \vdots  \\ \lambda_n  \end{array}  \right),&\, \text{if }  \lambda_1+\lambda_n>x_0 \\
                                                       T_1  \left( \begin{array}{c} \lambda_0 \\ \lambda_1 \\ \vdots  \\ \lambda_n  \end{array}  \right),&\, \text{if } \lambda_1+\lambda_n<x_0 \end{array} \right.  \\
                                                         \end{array}$$

  where
\[
    T_0 = \begin{pmatrix} 0 & 1 & 0  & \cdots & 0\\ 0 & 0 & 1 & \cdots & 0 \\  \vdots & \vdots & \vdots &\vdots & \vdots    \\0 & 0 & 0 & \cdots & 1  \\1 & -1 &
        0 & \cdots & 0 \end{pmatrix}
    \quad\text{and}\quad T_1= \begin{pmatrix} 1 & 0 & 0  & \cdots &  0 & -1\\ 0 & 1 &0&   \cdots &  0 & 0 \\  \vdots & \vdots & \vdots &\vdots & \vdots & \vdots  \\ 0 & 0 & 0 & \cdots & 0 & 1  \\\end{pmatrix}\]
Thus for $n=2$, we have 
\[
    T_0 = \begin{pmatrix} 0 & 1 & 0 \\ 0 & 0 & 1 \\ 1 & -1 &
        0 \end{pmatrix}
    \quad\text{and}\quad T_1= \begin{pmatrix} 1 & 0 & -1 \\ 0 & 1 & 0
        \\ 0&0&1 \end{pmatrix}.
\]

Note that 

\[
  T_0^{-1} =
    \begin{pmatrix}
        1 & 0 & 0&  \cdots & 0 &1 \\
        1 & 0 & 0 & \cdots & 0 & 0 \\ 0 &1&  0 &\cdots & 0 & 0 \\ 0&0&1& \cdots & 0 & 0\\
   \vdots & \vdots & \vdots &\vdots & \vdots & \vdots  \\
        0 & 0 & 0 & \cdots & 1 & 0\end{pmatrix} \quad\text{and}\quad
  T_1^{-1}= \begin{pmatrix} 1 & 0 & 0  & \cdots &  0 & 1\\ 0 & 1 &0&   \cdots &  0 & 0 \\ \vdots & \vdots & \vdots &\vdots & \vdots & \vdots \\ 0 & 0 & 0 & \cdots & 0 & 1  \\\end{pmatrix}
\]

Both  $T_0^{-1}$ and $T_1^{-1}$ have non-zero entries.  This is not true for most other multi-dimensional continued fractions.  It is this that allows us to use the triangle map to study integer partitions. 

The  domain of $T$ is:
 
 \begin{eqnarray*}
 \mbox{Domain}(T) &=&  \mbox{Dom}(T) \\
 &=& \{\overline{\lambda}   \in \triangle: \forall n\in \N, T^n (\overline{\lambda} \times \overline{k})\not\in \triangle_D \}
 \end{eqnarray*}
 (This is the analog in the world of traditional continued fractions  of removing all rational numbers, since the continued fraction expansion of any rational number is finite.)
 
 This in turn allows us to define the slow triangle sequence of a $\overline{\lambda}   \in   \mbox{Dom}(T) $:
 
 \begin{definition} The slow triangle sequence for any $\overline{\lambda}   \in   \mbox{Dom}(T) $ is the sequence $(i_0, i_1, i_2, \ldots)$ of zeros and ones such that 
 $$T^n(\overline{\lambda} )  \in   \triangle_{i_n}$$
 for all non-negative $n$.
 \end{definition}
It is known that if this sequence is eventually periodic, then the ratios $\lambda_i/\lambda_0$ are all algebraic numbers in the same algebraic number field of degree less than or equal to $m$, where
$$\overline{\lambda}   =(\lambda_0, \ldots, \lambda_m),$$
which was one of the prime motivations for the definition of $T$.

The cylinders of $T$ are key, and are 
$$\triangle_{i_0, \ldots , i_n} = \{ \overline{\lambda} \in   \mbox{Dom}(T):  T^k(\overline{\lambda} )  \in   \triangle_{i_k}, 0\leq k \leq n\}.$$
This gives us

\begin{prop} We have 
$$  T(\triangle_{i_0, \ldots , i_n} ) = \triangle_{i_1, \ldots , i_n} .$$
\end{prop}

Finally, note that the map $T$ is two-to-one, which is why we need to search for the natural extension.

\section{The homogeneous extended triangle map}\label{The homogeneous extended triangle map} 

\subsection{The map}
Fix an integer $m>1$.  
Set
$$\overline{\lambda} =  (\lambda_0, \ldots, \lambda_m), \overline{k} = (k_0, \ldots, k_m). $$
As before,   the $\lambda_i$ are a sequence of decreasing  positive real numbers but the $k_i$  are simply positive reals.
For any positive real $N$, set 

$$\mathcal{P}_m(N) = \{\overline{\lambda} \times \overline{k}: \overline{\lambda} \in \triangle, k_i>0, N = \sum_{i=0}^m \lambda_i k_i \} $$

Note that when $N$ is a positive integer, and when all the $\lambda_i$ and $k_j$ are positive integers, then the union of all the $\mathcal{P}_m(N)$ for all  $m\geq 1$  is the integer partition space for $N$ for more than one part.

Define
 \begin{eqnarray*}
 \tilde{\triangle}_0 &=& \{\overline{\lambda} \times \overline{k}\in \mathcal{P}_m(N) : \overline{\lambda} \in \triangle_0\} \\
  \tilde{\triangle}_1 &=& \{\overline{\lambda} \times \overline{k}\in \mathcal{P}_m(N) : \overline{\lambda} \in \triangle_1\} 
 \end{eqnarray*}

\begin{definition} The (homogeneous) extended triangle map 
$$\tilde{T}: \tilde{\triangle}_0 \cup \tilde{\triangle}_1 \rightarrow \tilde{\triangle}$$
are the maps
$$\begin{array}{c} (\lambda_0, \ldots , \lambda_m) \times [k_0,\ldots , k_m] \\ \\ \tilde{T}_0 \downarrow \\ \\ (\lambda_1, \lambda_2 \ldots , \lambda_m, \lambda_0-\lambda_1) \times [k_0+k_1,k_2, \ldots , k_m, k_0]  \end{array}$$
if $\lambda_1 + \lambda_m > \lambda_0$ (meaning $\overline{\lambda} \times \overline{k}\in \tilde{\triangle}_0$) and

$$\begin{array}{c} (\lambda_0, \ldots , \lambda_m) \times [k_0,\ldots , k_m] \\ \\ \tilde{T}_1 \downarrow \\ \\ (\lambda_0-\lambda_m, \lambda_1, \ldots , \lambda_m) \times [k_0, \ldots , k_{m-1}, k_0+ k_m]  \end{array}$$
if $\lambda_1 + \lambda_m < \lambda_0$ (meaning $\overline{\lambda} \times \overline{k}\in \tilde{\triangle}_1$) 

\end{definition}

The obvious question is why we are defining the action of $\tilde{T}_0$ and $\tilde{T}_1$ on the vector $\overline{k}$ in the way we do.
The answer  follows from that
\begin{eqnarray*}
N &=& \overline{k} \cdot \overline{\lambda}^{\top} \\
&=& \overline{k} T_i^{-1} \cdot T_i  \overline{\lambda}^{\top} 
\end{eqnarray*}
This means that each $\tilde{T}_i$ can be interpreted as the $2(m+1) \times 2(m+1)$ matrix
$$\left(  \begin{array}{cc} T_i & 0 \\ 0 & (T_i^{-1})^{\top} \end{array} \right) $$
acting on the column vector
$$ \left(  \begin{array}{c} \overline{\lambda}^{\top} \\     \overline{k}^{\top} \end{array}  \right) $$
For example, if $m=1$, we can interpret  $\tilde{T}_0$ as

 \begin{eqnarray*}\tilde{T_0} \left( \begin{array}{c} \lambda_0\\ \lambda_1  \\ k_0 \\ k_1 \end{array} \right) &=& \left( \begin{array}{cc} T_0& 0\\0 & (T_0^{-1})^{T} \end{array} \right)   \left( \begin{array}{c} \lambda_0\\  \lambda_1 \\ k_0 \\ k_1\end{array} \right) \\ &=& \left( \begin{array}{cccc} 0& 1&0&0\\1 & -1&0&0 \\ 0&0&1&1\\ 0&0& 1&0 \end{array} \right)   \left( \begin{array}{c} \lambda_0 \\\lambda_1 \\ k_0 \\ k_1\end{array} \right) \\&=&   \left( \begin{array}{c} \lambda_1\\  \lambda_0-\lambda_1 \\ k_0+k_1 \\ k_0 \end{array} \right) 
 \end{eqnarray*}
and $\tilde{T}_1$ as

 \begin{eqnarray*}\tilde{T_1} \left( \begin{array}{c} \lambda_0\\    \lambda_1   \\ k_0 \\ k_1 \end{array} \right) &=& \left( \begin{array}{cc} T_1& 0\\0 & (T_1^{-1})^{T} \end{array} \right)   \left( \begin{array}{c} \lambda_0\\    \lambda_1 \\ k_0 \\ k_1   \end{array} \right) \\ &=& \left( \begin{array}{cccc} 1& -1&0&0\\0 & 1&0&0 \\ 0&0&1&0\\ 0&0& 1&1 \end{array} \right)   \left( \begin{array}{c} \lambda_0\\\lambda_1 \\ k_0 \\ k_1\end{array} \right) \\&=&   \left( \begin{array}{c} \lambda_0-\lambda_1\\\lambda_1 \\ k_0 \\ k_0+k_1 \end{array} \right) 
 \end{eqnarray*}

 This allows us to define the domain of $\tilde{T}$:
 
 \begin{eqnarray*}
 \mbox{Domain}(\tilde{T}) &=&  \mbox{Dom}(\tilde{T}) \\
 &=& \{\overline{\lambda} \times \overline{k}\in \mathcal{P}_m(N): \forall n\in \N, \tilde{T}^n (\overline{\lambda} \times \overline{k})\not\in \triangle_D \}
 \end{eqnarray*}
 (This is the analog in the world of traditional continued fractions  of removing all rational numbers, since the continued fraction expansion of any rational number is finite.)

Now to start to understand the ranges of $\tilde{T}$.   Set
 \begin{eqnarray*}
    \tilde{\nabla}_0 &=& \{\overline{\lambda} \times \overline{k}\in \mathcal{P}_m(N) :  k_0 > k_m \} \\
     \tilde{\nabla}_1 &=& \{\overline{\lambda} \times \overline{k}\in \mathcal{P}_m(N) :  k_0 < k_m \} \\
 \tilde{\nabla}_D &=& \{\overline{\lambda} \times \overline{k}\in \mathcal{P}_m(N) :  k_0 =k_m \} \\
 \end{eqnarray*}

We have 
\begin{prop} The maps
$$\tilde{T}_0:  \tilde{\triangle}_0 \rightarrow  \tilde{\nabla}_0$$
and
$$\tilde{T}_1:  \tilde{\triangle}_1 \rightarrow  \tilde{\nabla}_1$$
are both one-to-one and onto maps.
\end{prop}
The proof is a calculation.

\subsection{The inverse of the homogeneous extended triangle map}
Here we simply write down the inverse maps for $ \tilde{T}_0$ and $\tilde{T}_1$.

The map 
$$\tilde{T}_0^{-1}:  \tilde{\nabla}_0 \rightarrow  \tilde{\triangle}_0$$
is 
$$\begin{array}{c} (\overline{\lambda} ) \times [\overline{k}] \\
\\
\downarrow  \tilde{T}_0^{-1} \\ \\
  (\lambda_0+\lambda_m, \lambda_0, \lambda_1, \ldots, \lambda_{m-1}) \times [k_m, k_0-k_m, k_1, k_2,\ldots , k_{m-1}]
  \end{array}$$
and the 
map 
$$\tilde{T}_1^{-1}:  \tilde{\nabla}_1 \rightarrow  \tilde{\triangle}_1$$
is 
$$\begin{array}{c} (\overline{\lambda} ) \times [\overline{k}] \\
\\
\downarrow  \tilde{T}_1^{-1} \\ \\
  (\lambda_0+\lambda_m, \lambda_1, \lambda_2, \ldots, \lambda_{m}) \times [k_0, k_1, k_2,\ldots , k_{m-1}, k_m-k_0].
  \end{array}$$

\subsection{The cylinders}\label{homog cylinders}

Given two sequences $(i_0, \ldots, i_{n_1})$ and $[j_0, \ldots, j_{n_2}]$, each consisting of ones and zeros, then define the corresponding  cylinder

$$\tilde{\triangle}_{ (i_0, \ldots, i_{n_1})\times [j_0, \ldots, j_{n_2}]   }  $$
to be all $\overline{\lambda} \times \overline{k} \in \tilde{\triangle}$ such that 
$$   \tilde{T}^k(\overline{\lambda} \times \overline{k}) \in \tilde{\triangle}_{i_k},   \tilde{T}^{-l}(\overline{\lambda} \times \overline{k}) \in \tilde{\nabla}_{j_l},$$
for any  integers $k$ and $l$ with $0\leq k \leq n_1$ and any $0\leq l \leq n_2.$

We have 

\begin{prop}\label{homog cylinders}   For all integers  $k$ with $0\leq k \leq n_1$, we have 
$$T^{k}\tilde{\triangle}_{  (i_0, \ldots, i_{n_1})\times (j_0, \ldots, j_{n_2})   } =   \tilde{\triangle}_{ (i_{k+1},  \ldots, i_{n_1}) \times  (i_k, i_{k-1}, \ldots, i_{k_0}, j_0, \ldots, j_{n_2} ) }  $$
\end{prop}

This will allow us to claim:
\begin{theorem}\label{homog natural exension} The map $\tilde{T}$ is the natural extension of the map $T$.

\end{theorem}
  In the four needed conditions for the natural extension given in Section \ref{basics of natural extension}, conditions 1, 2 and 4 are what we have been discussing, and condition 2 follows from Proposition \ref{homog cylinders}.  The shortness of this proof reflects that the heart of the argument is finding the natural extension, not the proof.  

\section{The Affine Projection}\label{affine projection}

\subsection{The Teichm\"{u}ller flow}
On the space $\mathcal{P}_m(N)$ there is the Teichm\"{u}ller flow
$$\mbox{Teich}_{s}:  \mathcal{P}_m(N) \rightarrow   \mathcal{P}_m(N) $$
defined by 
$$ \mbox{Teich}_s (\overline{\lambda}, \overline{k}) = (e^s \overline{\lambda}, e^{-s}\overline{k}) $$
(Though not relevant to this paper, this map is critical in the study of translation surfaces, which is one of the  hidden  motivations for this paper.  See Athreya and Masur \cite{Athreya-Masur}.)

We  choose $s$ so that $e^s = 1/\lambda_0.$  For this fixed $s$, we define
\begin{eqnarray*}
\mbox{Teich} (\overline{\lambda} \times \overline{k})  &=&  ((1/\lambda_0) \overline{\lambda}\times  \lambda \overline{k}) \\
&=& \left( 1, \frac{\lambda_1}{\lambda_0}, \ldots,  \frac{\lambda_m}{\lambda_0} \right) \times \left[ \lambda_0 k_0, \lambda_0 k_1, \ldots, \lambda_0 k_m\right].
\end{eqnarray*}
The Teichm\"{u}ller flow will be what we use to get next subsection's affine projections.

\subsection{The map}

The triangle map $T$ and its natural extension $\tilde{T}$ depend not so much on the  vector $\overline{\lambda} = (\lambda_0, \ldots, \lambda_m)$ but actually on the various ratios of  $\lambda_i/\lambda_0.$ 
This suggests that we should be thinking of $\overline{\lambda} $ as a point in projective space.  As $\lambda_0\neq 0$, this in turn suggests the map
$$(\lambda_0, \ldots, \lambda_m) \rightarrow  \left(\frac{\lambda_1}{\lambda_0}, \ldots,  \frac{\lambda_m}{\lambda_0} \right),$$
which can be interpreted as part of a Teichm\"{u}ller flow.

In turn, this leads to the affine projection map
$$\pi:\mathcal{P}_m (N) \rightarrow \R^m \times \R^m$$
defined as
$$\pi (\overline{\lambda} \times \overline{k}) 
=\left( \frac{\lambda_1}{\lambda_0}, \ldots,  \frac{\lambda_m}{\lambda_0} \right) \times \left[  \lambda_0 k_1, \ldots, \lambda_0 k_m\right].$$

Set 
$$ \tilde{\triangle}_m^A(N) = \{ (t_1, \ldots, t_m)\times [u_1, \ldots, u_m]\in \R^m\times \R^m: t_1 > \cdots > t_m>0, u_i>0,  N- \sum_{i=1}^m t_i u_i >0\}.$$
This is the affine version of $\mathcal{P}_m(N)$, as seen in the following proposition:

\begin{prop} We have that the affine projection $\pi$ maps $\mathcal{P}_m (N) $ onto the space $ \tilde{\triangle}_m^A(N) $.
\end{prop}

\begin{proof} 

The affine projection map is the composition of 
$$\begin{matrix}
(\lambda_0, \ldots, \lambda_m) \times [k_0, \ldots, k_m]\\
\downarrow\\
 \left( 1, \frac{\lambda_1}{\lambda_0}, \ldots,  \frac{\lambda_m}{\lambda_0} \right) \times \left[ \lambda_0 k_0, \lambda_0 k_1, \ldots, \lambda_0 k_m\right]\\
 \downarrow \\
 \left( \frac{\lambda_1}{\lambda_0}, \ldots,  \frac{\lambda_m}{\lambda_0} \right) \times \left[  \lambda_0 k_1, \ldots, \lambda_0 k_m\right]
\end{matrix}$$
Given an element $$(t_1, \ldots, t_m)\times [u_1, \ldots, u_m] \in \triangle_m^A(N)$$
we can set 
$$\lambda_0=1, \lambda_1=t_1, \ldots, \lambda_m=t_m, k_1=u_1,\ldots, k_m= u_m.$$
The information that is missing is the value of $k_0$.  But we need
$$\lambda_0 k_0 + \lambda_1k_1 + \ldots + \lambda_mk_m = N,$$
meaning that we need
$$k_0 = N - \sum_{i=1}^m \lambda_i k_i  = N - \sum_{i=1}^m t_i u_i . $$
This we can do if the above is positive, which explains why we defined $ \tilde{\triangle}_m^A(N) $ as we did.
 \end{proof}

This will allow us to find the nonhomogeneous extended triangle map, which we will also call the affine extended triangle map.
The non-homogeneous triangle map has two subdomains:

\begin{eqnarray*}
\tilde{\triangle}_0^A &=& \{  \overline{t}  \times \overline{u}     \in \tilde{\triangle} : 1<t_1+t_m\}  \\
\tilde{\triangle}_1^A &=& \{ \overline{t}  \times \overline{u}   \in \tilde{\triangle} : 1>t_1+t_m\}  \\
\end{eqnarray*}
In both cases we are suppressing a subscript of an $m$.

Via a calculation, we have 
\begin{prop} The affine projection map $\pi$ maps $\tilde{\triangle}_0$ onto $  \tilde{\triangle}_0^A$ and maps  $\tilde{\triangle}_1$ onto $  \tilde{\triangle}_1^A.$
\end{prop}

The natural extended non-homogeneous triangle map is defined via two maps
$$\tilde{T}_0^A:\tilde{\triangle}_0^A \rightarrow \tilde{\triangle}^A, \; \tilde{T}_1:\tilde{\triangle}_1^A \rightarrow \tilde{\triangle}^A,$$
where we set 
$$\begin{array}{c}
(t_1, \ldots, t_m) \times  [u_1, \ldots, u_m)] \\ \\
\downarrow \tilde{T}_0^A \\ \\
\left(   \frac{t_2}{t_1}, \ldots, \frac{t_m}{t_1}, \frac{1-t_1}{t_1}        \right) \times \left[  t_1 u_2, \ldots, t_1u_m, t_1 \left( N- \sum_{i=1}^m t_i u_i \right)   \right] \\
\end{array}$$
and 
$$\begin{array}{c}
(t_1, \ldots, t_m) \times  [u_1, \ldots, u_m)] \\ \\
\downarrow \tilde{T}_1^A \\ \\
\left(  \frac{t_1}{1-t_m}, \ldots,  \frac{t_m}{1-t_m}  \right) \times \left[  (1-t_m) u_1, \ldots, (1-t_m) u_{m-1} , (1-t_m)  \left( N+ u_m- \sum_{i=1}^m t_i u_i \right)   \right]
\end{array}$$

The justification for defining these maps in the way we did is in:

\begin{prop} The following two diagrams are commutative:
$$\begin{matrix} 
\tilde{\triangle}_0 & \xrightarrow{\tilde{T}_0} & \tilde{\triangle} \\
\pi \downarrow & & \downarrow \pi \\
\tilde{\triangle}_0^A & \xrightarrow{\tilde{T}_0^A} & \tilde{\triangle}^A 
\end{matrix}$$
and
$$\begin{matrix} 
\tilde{\triangle}_1 & \xrightarrow{\tilde{T}_1} & \tilde{\triangle} \\
\pi \downarrow & & \downarrow \pi \\
\tilde{\triangle}_1^A & \xrightarrow{\tilde{T}_1^A} & \tilde{\triangle}^A 
\end{matrix}$$
\end{prop} 
\begin{proof}The proof is a long but straightforward calculation. \end{proof}

\subsection{The inverses}

We can calculate the inverses to be:

$$\begin{array}{c}
(s_1, \ldots, s_m) \times  [v_1, \ldots, v_m)] \\ \\
\downarrow \tilde{T}_0^{-1} \\ \\
\left(   \frac{1}{1+s_m}   ,    \frac{s_1}{1+s_m}   ,  \frac{s_2}{1+s_m}   \ldots,   ,    \frac{s_{m-1}}{1+s_m}       \right)\\
 \times \\
  \left[ (1+s_m) \left(  N- \sum_{n=1}^{m-1} s_n v_n - (1+s_m) v_m   \right) , (1+s_m) v_1, (1+s_m)v_2, \ldots, (1+s_m) v_{m-1}  \right] \\
\end{array}$$
and
$$\begin{array}{c}
(s_1, \ldots, s_m) \times  [v_1, \ldots, v_m)] \\ \\
\downarrow \tilde{T}_1^{-1} \\ \\
\left(   \frac{s_1}{1+s_m}   ,  \frac{s_2}{1+s_m}     , \ldots  \frac{s_m}{1+s_m}      \right) \\
 \times \\
 \left[   (1+s_m)v_1, \ldots, (1+s_m)v_{m-1}, (1+s_m) \left( (1+s_m)v_m + \left( \sum_{n=1}^{m-1} s_nv_n\right)  - N \right)      \right]\\
\end{array}.$$

\subsection{The cylinders}

Finally, we want to write down the various domains and co-domains for the non-homogeneous triangle map.  We already know the two domains $ \tilde{\triangle}_0^A $ and $\tilde{\triangle}_1^A $.

The two co-domains will be 
\begin{eqnarray*}
 \tilde{\nabla}_0^A &=& \{\overline{ t } \times \overline{u}   \in \pi(\mathcal{P}_m(N)) : N- \sum_{i-1}^m t_i  k_i  > k_m \} \\
     \tilde{\nabla}_1^A &=& \{\overline{   t    } \times \overline{u}\in  \pi(\mathcal{P}_m(N))  :   N- \sum_{i-1}^m t_i  k_i  <k_m \} \\
\end{eqnarray*}
As before,  we are suppressing a subscript of an $``m"$.

The affine cylinders notation is almost exactly the same as  in subsection \ref{homog cylinders}.
Given two sequences $(i_0, \ldots, i_{n_1})$ and $[j_0, \ldots, j_{n_2}]$, each consisting of ones and zeros, we define the corresponding  cylinder

$$\tilde{\triangle}_{ (i_0, \ldots, i_{n_1})\times [j_0, \ldots, j_{n_2}]   }^A  $$
to be all $\overline{\lambda} \times \overline{k} \in \tilde{\triangle}^A$ such that 
$$   \tilde{T}^k(\overline{\lambda} \times \overline{k}) \in \tilde{\triangle}_{i_k}^A,   \tilde{T}^{-l}(\overline{\lambda} \times \overline{k}) \in \tilde{\nabla}_{j_l}^A,$$
for any  integers $k$ and $l$ with $0\leq k \leq n_1$ and any $0\leq l \leq n_2.$

As before, we have 

\begin{prop} For all integers  $k$ with $0\leq k \leq n_1$, we have 
$$\tilde{T}^{k}\tilde{\triangle}_{  (i_0, \ldots, i_{n_1})\times (j_0, \ldots, j_{n_2})   }^A =   \tilde{\triangle}_{ (i_{k+1},  \ldots, i_{n_1}) \times  (i_k, i_{k-1}, \ldots, i_{k_0}, j_0, \ldots, j_{n_2} ) }^A  $$
\end{prop}

As with Theorem \ref{homog natural exension} , this will allow us to claim:
\begin{theorem}   The map $\tilde{T}$ is the natural extension of the map $T$.

\end{theorem}
The proof is the same as the earlier theorem.

\section{``Graphs'' of natural extensions}\label{``Graphs'' of Natural Extensions}

\subsection{The classical Farey map:  $m=1$ case}
Let us look at the case when $m=1.$ This case is classical, and is implicit in Arnoux and Nogueira \cite{Arnoux-Nogueira-93}.

Fix a positive number $N$. Set

$$ \tilde{\triangle}^A(N)  = \{ (t) \times (u)  :1\geq  t \geq 0, u\geq 0, N  \geq tu \}.$$
The advantage of $m=1$ is that this is a domain in $\R^2,$ and hence we can actually draw it. 

The two  slow non-homogeneous triangle map subdomains are :

\begin{eqnarray*}
\tilde{\triangle}_0^A &=& \{ (t)\times (u)     \in \tilde{\triangle}^A: 1/2< t  \}  \\
\tilde{\triangle}_1^A &=& \{ (t) \times (u)   \in \tilde{\triangle}^A : 1/2>t\}  \\
\end{eqnarray*}
(When $m=1$, the term $t_1+ t_m$ becomes $t_1+t_1.$)

The two co-domains now will be 
\begin{eqnarray*}
 \tilde{\nabla}_0^A &=& \{   (t) \times  (u)   \in \pi(\mathcal{P}_m(N)) : N- tu > u \} \\
 &=&  \{   (t) \times  (u)   \in \pi(\mathcal{P}_m(N)) : \frac{N}{1+t}> u \} \\
     \tilde{\nabla}_1^A &=& \{    (t)  \times   (u)   \in   \pi(\mathcal{P}_m(N))  :   N-  t u  < u  \} \\
      &=&  \{   (t) \times  (u)   \in \pi(\mathcal{P}_m(N)) : \frac{N}{1+t}< u \} 
\end{eqnarray*}

These we can actually draw (where we are letting $N=1$):

$$\begin{array}{cc}

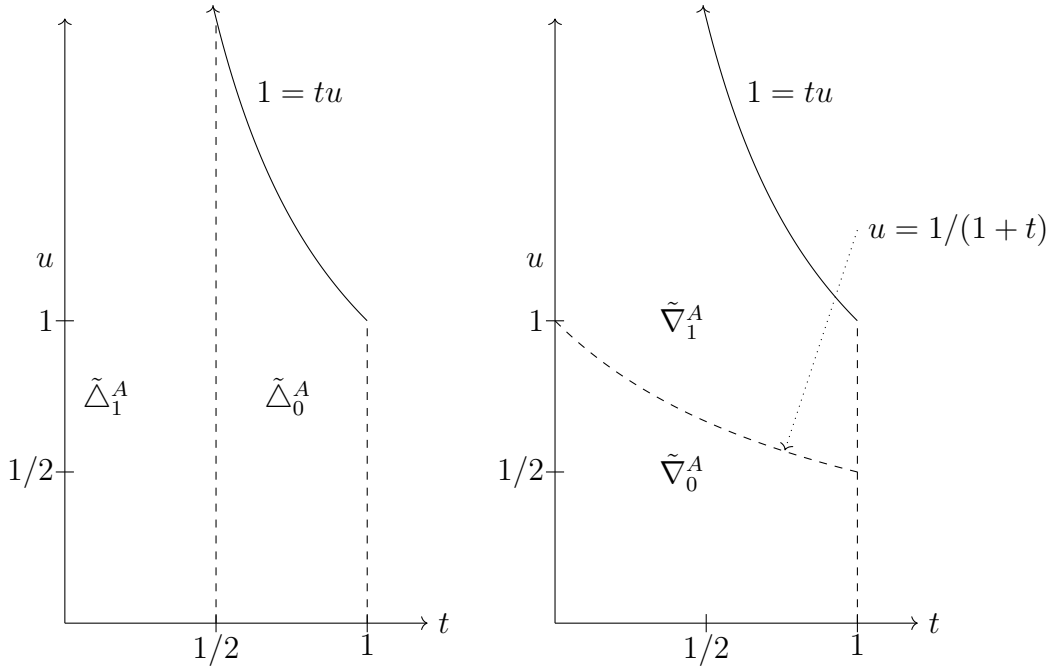
\begin{tikzpicture}[scale=4]

\draw[->](0,0)--(1.2, 0);
\draw[->](0,0)--(0, 2);
\draw[dashed](1,0)--(1, 1);
\draw[dashed](1/2,0)--(1/2, 2);

\node[right] at (1.2,0){$t$};
\node[left]at (0,1.2){$u$};

\draw[-](1,.03)--(1, -.03);
\node[below] at (1,0){$1$};

\draw[-](.5,.03)--(.5, -.03);
\node[below] at (.5,0){$1/2$};

\node[left] at (0,1){$1$};

\draw[-](.03, 1)--(-.03, 1);
\draw[-](.03, 1/2)--(-.03, 1/2);

\node[left] at (0,1/2){$1/2$};

\draw[scale=1, domain=.49:1, smooth, variable=\x] plot ({\x}, {1/(  \x)});

\draw[->](.49,1/.49)--(.49, 1/.49);

\node[left] at (.85,.75){$ \tilde{\triangle}_0^A$};
\node[left] at (.25,.75){$ \tilde{\triangle}_1^A$};

\node[left] at (1,1.75){$\begin{array}{c} 1=tu \end{array}$};

\end{tikzpicture}

& 
\begin{tikzpicture}[scale=4]

\draw[->](0,0)--(1.2, 0);
\draw[->](0,0)--(0, 2);
\draw[dashed](1,0)--(1, 1);

\node[right] at (1.2,0){$t$};
\node[left]at (0,1.2){$u$};

\draw[-](1,.03)--(1, -.03);
\node[below] at (1,0){$1$};

\draw[-](.5,.03)--(.5, -.03);
\node[below] at (.5,0){$1/2$};

\draw[-](.03, 1)--(-.03, 1);
\node[left] at (0,1){$1$};

\draw[-](.03, 1/2)--(-.03, 1/2);

\node[left] at (0,1/2){$1/2$};

\draw[scale=1, domain=.49:1, smooth, variable=\x] plot ({\x}, {1/(  \x)});

\draw[->](.49,1/.49)--(.49, 1/.49);

\draw[scale=1, domain=0:1, smooth, variable=\x, dashed] plot ({\x}, {1/( 1+  \x)});

\node[left] at (.53,1){$ \tilde{\nabla}_1^A$};
\node[left] at (.53,.5){$ \tilde{\nabla}_0^A$};

\node[left] at (1,1.75){$\begin{array}{c} 1=tu \end{array}$};

\draw[dotted,->](1,1.3)--(.76, 1/1.75);

\node[right] at (1, 1.3){$ u=1/(1+t)$};

\end{tikzpicture}

\end{array} $$

We have 

$$\begin{array}{c}
(t) \times  [u] \\
\downarrow \tilde{T}_0 \\
\left(    \frac{1-t}{t}        \right) \times \left[ t \left( N- tu\right)   \right] \\
\end{array}$$
and 
$$\begin{array}{c}
(t \times  [u] \\
\downarrow \tilde{T}_1 \\
\left(   \frac{t}{1-t}  \right) \times \left[  (1-t)  \left( N+ u- tu \right)   \right] \\

\end{array}$$

The cylinders $\tilde{\triangle}_{(0) \times [0]}^A, \tilde{\triangle}_{(0) \times [1]}^A, \tilde{\triangle}_{(1) \times [0]}^A $ and $\tilde{\triangle}_{(1) \times [1]}^A  $ are

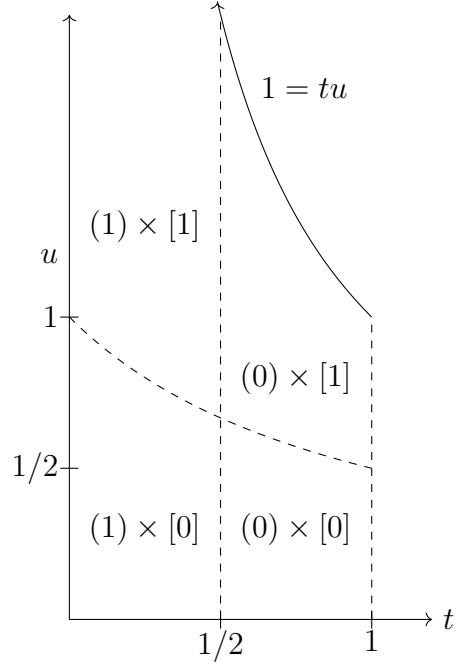
\begin{center}
\begin{tikzpicture}[scale=4]

\draw[->](0,0)--(1.2, 0);
\draw[->](0,0)--(0, 2);
\draw[dashed](1,0)--(1, 1);
\draw[dashed](1/2,0)--(1/2, 2);

\node[right] at (1.2,0){$t$};
\node[left]at (0,1.2){$u$};

\draw[-](1,.03)--(1, -.03);
\node[below] at (1,0){$1$};

\draw[-](.5,.03)--(.5, -.03);
\node[below] at (.5,0){$1/2$};

\draw[-](.03, 1)--(-.03, 1);
\node[left] at (0,1){$1$};

\draw[-](.03, 1/2)--(-.03, 1/2);

\node[left] at (0,1/2){$1/2$};

\draw[scale=1, domain=.49:1, smooth, variable=\x] plot ({\x}, {1/(  \x)});

\draw[->](.49,1/.49)--(.49, 1/.49);

\draw[scale=1, domain=0:1, smooth, variable=\x, dashed] plot ({\x}, {1/( 1+  \x)});

\node[left] at (.47,1.3){$(1)\times [1]$};
\node[left] at (.47,.3){$(1)\times [0]$};
\node[left] at (.97,.3){$(0)\times [0]$};
\node[left] at (.97,.8){$(0)\times [1]$};

\node[left] at (1,1.75){$\begin{array}{c} 1=tu \end{array}$};

\

\end{tikzpicture}
\end{center}
where we are letting $(i)\times [j]$ stand for $\tilde{\triangle}_{(i) \times [j]}^A$ in order to reduce clutteredness in the diagram.

\subsection{``Graphs'' for higher dimensional Natural Extensions}
Unlike in the non-homogenous $m=1$ case, we cannot graph the natural extensions in higher dimensions.  But we can represent them schematically.

In the non-homogeneous case we have

$$\begin{array}{cc}

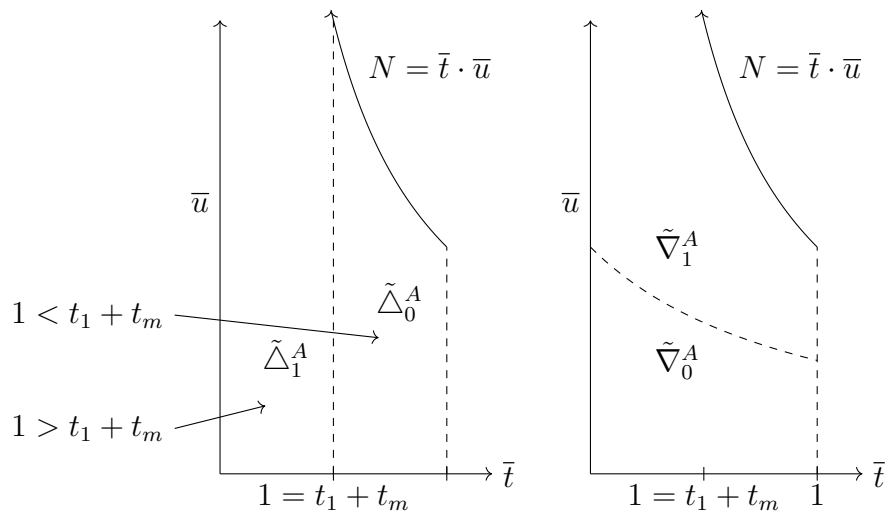
\begin{tikzpicture}[scale=3]

\draw[->](0,0)--(1.2, 0);
\draw[->](0,0)--(0, 2);
\draw[dashed](1,0)--(1, 1);
\draw[dashed](1/2,0)--(1/2, 2);

\node[right] at (1.2,0){$\overline{t}$};
\node[left]at (0,1.2){$\overline{u}$};

\draw[-](1,.03)--(1, -.03);
\draw[-](.5,.03)--(.5, -.03);
\node[below] at (.5,0){$1=t_1+t_m$};

;

\draw[scale=1, domain=.49:1, smooth, variable=\x] plot ({\x}, {1/(  \x)});

\draw[->](.49,1/.49)--(.49, 1/.49);

\node[left] at (1,.75){$\begin{array}{c}  \tilde{\triangle}_0^A\\ \end{array}$};
\draw[->](-.2,.7)--(.7, .6);
\node[left] at (-.2,.7){$1<t_1+t_m$};

\node[left] at (.5,.5){$\begin{array}{c}  \tilde{\triangle}_1^A\\  \end{array}$};
\draw[->](-.2,.2)--(.2, .3);
\node[left] at (-.2,.2){$1>t_1+t_m$};

\node[left] at (1.3,1.78){$\begin{array}{c} N=\overline{t} \cdot \overline{u} \end{array}$};

\end{tikzpicture}

& 
\begin{tikzpicture}[scale=3]

\draw[->](0,0)--(1.2, 0);
\draw[->](0,0)--(0, 2);
\draw[dashed](1,0)--(1, 1);

\node[right] at (1.2,0){$\overline{t}$};
\node[left]at (0,1.2){$\overline{u}$};

\draw[-](1,.03)--(1, -.03);
\node[below] at (1,0){$1$};

\draw[-](.5,.03)--(.5, -.03);
\node[below] at (.5,0){$1= t_1+t_m$};

\node[left] at (1.3,1.78){$\begin{array}{c} N=\overline{t} \cdot \overline{u} \end{array}$};

\draw[scale=1, domain=.49:1, smooth, variable=\x] plot ({\x}, {1/(  \x)});

\draw[->](.49,1/.49)--(.49, 1/.49);

\draw[scale=1, domain=0:1, smooth, variable=\x, dashed] plot ({\x}, {1/( 1+  \x)});

\node[left] at (.53,1){$ \tilde{\nabla}_1^A$};
\node[left] at (.53,.5){$ \tilde{\nabla}_0^A$};

\end{tikzpicture}

\end{array} $$

The same diagrams (with slightly different labelings) can be used for the natural extensions in the homogeneous case:

$$\begin{array}{cc}

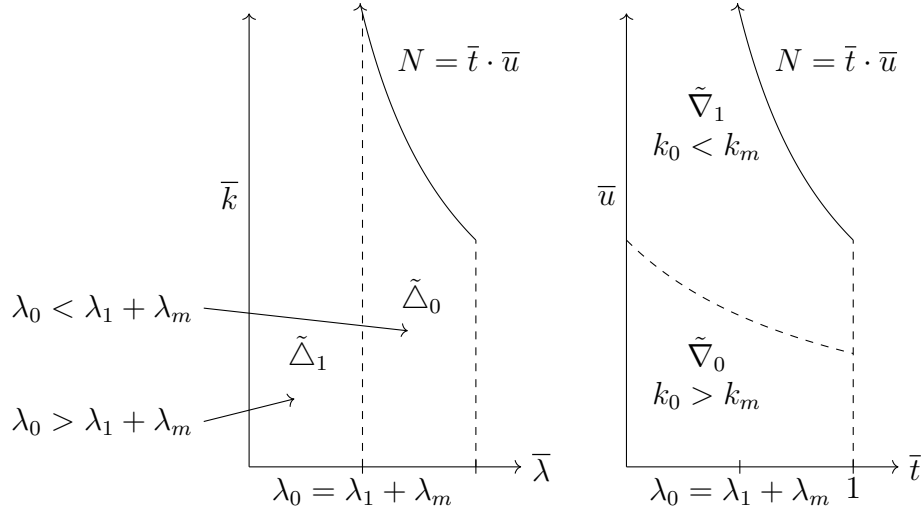
\begin{tikzpicture}[scale=3]

\draw[->](0,0)--(1.2, 0);
\draw[->](0,0)--(0, 2);
\draw[dashed](1,0)--(1, 1);
\draw[dashed](1/2,0)--(1/2, 2);

\node[right] at (1.2,0){$\overline{\lambda}$};
\node[left]at (0,1.2){$\overline{k}$};

\draw[-](1,.03)--(1, -.03);
\draw[-](.5,.03)--(.5, -.03);
\node[below] at (.5,0){$\lambda_0=\lambda_1+\lambda_m$};

;

\draw[scale=1, domain=.49:1, smooth, variable=\x] plot ({\x}, {1/(  \x)});

\draw[->](.49,1/.49)--(.49, 1/.49);

\node[left] at (.9,.75){$ \tilde{\triangle}_0$};
\draw[->](-.2,.7)--(.7, .6);
\node[left] at (-.2,.7){$\lambda_0<\lambda_1+\lambda_m$};

\node[left] at (.4,.5){$ \tilde{\triangle}_1$};
\draw[->](-.2,.2)--(.2, .3);
\node[left] at (-.2,.2){$\lambda_0>\lambda_1+\lambda_m$};

\node[left] at (1.3,1.78){$\begin{array}{c} N=\overline{t} \cdot \overline{u} \end{array}$};

\end{tikzpicture}

& 
\begin{tikzpicture}[scale=3]

\draw[->](0,0)--(1.2, 0);
\draw[->](0,0)--(0, 2);
\draw[dashed](1,0)--(1, 1);

\node[right] at (1.2,0){$\overline{t}$};
\node[left]at (0,1.2){$\overline{u}$};

\draw[-](1,.03)--(1, -.03);
\node[below] at (1,0){$1$};

\draw[-](.5,.03)--(.5, -.03);
\node[below] at (.5,0){$\lambda_0= \lambda_1+\lambda_m$};

\node[left] at (1.3,1.78){$\begin{array}{c} N=\overline{t} \cdot \overline{u} \end{array}$};

\draw[scale=1, domain=.49:1, smooth, variable=\x] plot ({\x}, {1/(  \x)});

\draw[->](.49,1/.49)--(.49, 1/.49);

\draw[scale=1, domain=0:1, smooth, variable=\x, dashed] plot ({\x}, {1/( 1+  \x)});

\node[left] at (.7,1.5){$\begin{array}{c} \tilde{\nabla}_1\\ k_0<k_m \end{array}$};
\node[left] at (.7,.4){$\begin{array}{c} \tilde{\nabla}_0\\ k_0>k_m \end{array}$};

\end{tikzpicture}

\end{array} $$

As before the cylinders $\tilde{\triangle}_{(0) \times [0]}^A, \tilde{\triangle}_{(0) \times [1]}^A, \tilde{\triangle}_{(1) \times [0]}^A $ and $\tilde{\triangle}_{(0) \times [0]}^A  $ can be captured as
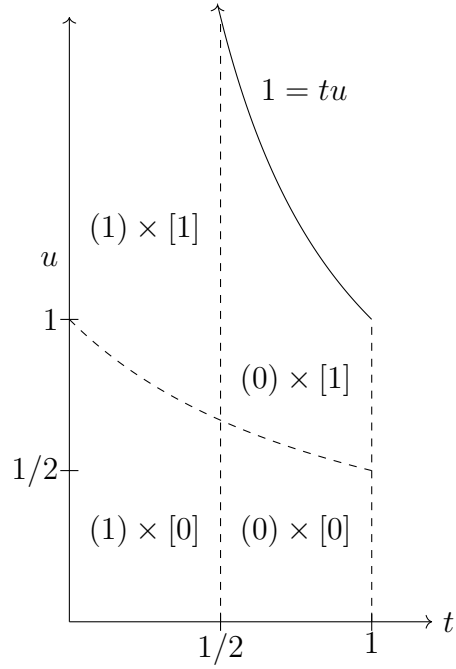
\begin{center}
\begin{tikzpicture}[scale=4]

\draw[->](0,0)--(1.2, 0);
\draw[->](0,0)--(0, 2);
\draw[dashed](1,0)--(1, 1);
\draw[dashed](1/2,0)--(1/2, 2);

\node[right] at (1.2,0){$t$};
\node[left]at (0,1.2){$u$};

\draw[-](1,.03)--(1, -.03);
\node[below] at (1,0){$1$};

\draw[-](.5,.03)--(.5, -.03);
\node[below] at (.5,0){$1/2$};

\draw[-](.03, 1)--(-.03, 1);
\node[left] at (0,1){$1$};

\draw[-](.03, 1/2)--(-.03, 1/2);

\node[left] at (0,1/2){$1/2$};

\draw[scale=1, domain=.49:1, smooth, variable=\x] plot ({\x}, {1/(  \x)});

\draw[->](.49,1/.49)--(.49, 1/.49);

\draw[scale=1, domain=0:1, smooth, variable=\x, dashed] plot ({\x}, {1/( 1+  \x)});

\node[left] at (.47,1.3){$(1)\times [1]$};
\node[left] at (.47,.3){$(1)\times [0]$};
\node[left] at (.97,.3){$(0)\times [0]$};
\node[left] at (.97,.8){$(0)\times [1]$};

\node[left] at (1,1.75){$\begin{array}{c} 1=tu \end{array}$};

\

\end{tikzpicture}
\end{center}
where again we are letting $(i)\times [j]$ stand for $\tilde{\triangle}_{(i) \times [j]}^A$.




\section{Young conjugation: an internal symmetry}\label{Young conjugation: An Internal Symmetry}

\subsection{Young conjugate map}

As mentioned in the introduction, a large part of motivation is that the triangle map can be used to study integer partitions.  But integer partitions can in turn be used to motivate properties of the natural extension. 

We can write any integer partition $$(\lambda_0, \ldots \lambda_m) \times [k_0, \ldots , k_m]$$
where all the terms are positive integers and 
$$\lambda_0>\lambda_1 > \cdots > \lambda_m $$
via its Young shape, which is $k_0$ horizontal  rows of  $\lambda_0$ boxes placed above $k_1$ horizontal  rows of  $\lambda_1$ boxes, etc.

Thus we can describe 
$$(5,3, 1) \times [2, 4, 5]$$ as

 \[
        \yng(5,5,3,3,3,3,1,1,1,1,1)
    \]

Just looking at this diagram, we can see that there is a clear symmetry: flip the diagram about the diagonal, to get

 \[
       \yng(11,6,6,2,2)
    \]
  which in coordinates is 
  $$(11, 6, 2) \times [1,2,2].$$
  Both are partitions of $27$.
  
  This ``flip'' is called Young conjugation:
  
  \begin{definition} Young conjugation is the map 
  $$\begin{array}{c}  (\lambda_0, \ldots \lambda_m) \times [k_0, \ldots , k_m] \\ \\
  \mathcal{Y} \downarrow \\ \\
   \left( \sum_{i=0}^m  k_i,  \sum_{i=0}^{m-1}  k_i, \ldots,k_0\right) \times [\lambda_m, \lambda_{m-1}-\lambda_m, \lambda_{m-2}- \lambda_{m-1}, 
  \ldots , \lambda_0 - \lambda_1]   \end{array}   $$
    \end{definition}
    
    Now we have that 
   $$(\overline{\lambda})\times [\overline{k}] \vdash N$$
   if and only if 
   $$\mathcal{Y}\left((\overline{\lambda})\times [\overline{k}] \right) \vdash N$$

   Though the Young shape only makes sense if all of the $\lambda_i$ and the $ k_j$ are positive integers, the map $\mathcal{Y}$ makes sense for any vector in 
   $\R^{m+1} \times \R^{m+1}.  $
Recalling that we defined $$\mathcal{P}_m(N) = \{\overline{\lambda} \times \overline{k}: \overline{\lambda} \in \triangle, k_i>0, N = \sum_{i=0}^m \lambda_i k_i \} $$
we have 

\begin{prop} The map
$$\mathcal{Y}: \mathcal{P}_m(N) \rightarrow \mathcal{P}_m(N)$$
is a one-to-one onto diffeomorphism.
\end{prop}
The proof is a simple calculation.

As we will later see, particularly relevant are the points in  $\mathcal{P}_m(N)$ which are fixed by $\mathcal{Y}.$
For notations, set 
$$\mathcal{F} = \{ (\overline{\lambda}) \times [\overline{k}] \in \mathcal{P}_m(N): \mathcal{Y} ( (\overline{\lambda}) \times [\overline{k}] ) =  (\overline{\lambda}) \times [\overline{k}] \}.$$

A direct calculation shows  

\begin{prop} We have that  $ (\overline{\lambda}) \times [\overline{k}] \in \mathcal{F}$ if and only if 
\begin{eqnarray*}
\lambda_0 &=& k_0 + \ldots + k_{ m  }  \\
\lambda_1 &=& k_0 + \ldots + k_{ m  -1}  \\
& \vdots &   \\
\lambda_{m-1} &=& k_0 + k_1 \\
\lambda_m &=& k_0
\end{eqnarray*}
\end{prop}

\subsection{The internal symmetry}

The Young conjugation map creates for us additional structure on the natural extension's domain $\mathcal{P}_m(N)$.

First, we have

\begin{prop}\label{initial symmetry}
\begin{eqnarray*}
 \tilde{\triangle}_0 &\xrightarrow{\mathcal{Y}} &   \tilde{\nabla}_0 \\
\tilde{\triangle}_1&\xrightarrow{\mathcal{Y}} &\tilde{\nabla}_1 \\
 \tilde{\nabla}_0 &\xrightarrow{\mathcal{Y}} &  \tilde{\triangle}_0 \\
\tilde{\nabla}_1 &\xrightarrow{\mathcal{Y}} &  \tilde{\triangle}_1 \\
\end{eqnarray*}
\end{prop}

The proofs are calculations, similar to the ones that we will do in a moment in the proof of Proposition \ref{internal symmetry}.

Unlike all other multi-dimensional continued fraction algorithms (that we know of), the triangle map and Young conjugation are quite compatible, a compatibility captured in the following:

\begin{prop} These two diagrams commute:

$$\begin{array}{ccccc}
  \begin{array}{ccc}
  \tilde{\triangle}_0 & \xrightarrow{\mathcal{Y}} & \tilde{\nabla}_0 \\
  T_0 \downarrow && \uparrow T_0 \\
   \tilde{\nabla}_0 & \xrightarrow{\mathcal{Y}} & \tilde{\triangle}_0
  \end{array}

& &&&
 \begin{array}{ccc}
  \tilde{\triangle}_1 & \xrightarrow{\mathcal{Y}} & \tilde{\nabla}_1 \\
  T_1 \downarrow && \uparrow T_1 \\
   \tilde{\nabla}_1 & \xrightarrow{\mathcal{Y}} & \tilde{\triangle}_1
  \end{array}
\end{array} $$

\end{prop}

\begin{proof}
Let $$(\overline{\lambda} )\times [ \overline{k} ] \in  \tilde{\triangle}_0 .$$
We will simply compute $T_0\circ \mathcal{Y} \circ T_0((\overline{\lambda} )\times [ \overline{k} ] )$ and that this is $\mathcal{Y}((\overline{\lambda} )\times [ \overline{k} ] ).$
We have 

$$\begin{array}{c} 
(\overline{\lambda} )\times [ \overline{k} ] \in  \tilde{\triangle}_0 \\ \\
T_0 \downarrow \\ \\
(\lambda_1, \ldots, \lambda_m, \lambda_0-\lambda_1) \times [ k_0+k_1, k_2, \ldots, k_{m}, k_0] \\ \\
\mathcal{Y} \downarrow \\ \\
(2k_0 + k_1 + \ldots k_m, k_0 + k_1 + \ldots k_m,  k_0 + \ldots k_m, \ldots, k_0+k_1] \\
\times \\
 \left[\lambda_0-\lambda_1, \lambda_m-(\lambda_0-\lambda_1), \lambda_{m-1}-\lambda_m, \ldots, \lambda_1-\lambda_2\right]\\ \\
 T_0\downarrow \\ \\
 (k_0+\ldots + k_m, \ldots, k_0+k_1, k_0) \times [\lambda_m, \lambda_{m-1}-\lambda_m, \ldots, \lambda_0-\lambda_1]\\ \\
 =\\ \\
 \mathcal{Y}((\overline{\lambda} )\times [ \overline{k} ] ).
 \end{array}$$
 as desired.
 
 The $T_1$ case is similar.
\end{proof}

What is important for us is that Proposition \ref{initial symmetry}
 can be refined quite a bit:

\begin{prop}\label{internal symmetry}
\begin{eqnarray*}
 \tilde{\triangle}_0 \cap  \tilde{\nabla}_0 &\xrightarrow{\mathcal{Y}} &  \tilde{\triangle}_0 \cap  \tilde{\nabla}_0 \\
\tilde{\triangle}_0 \cap  \tilde{\nabla}_1 &\xrightarrow{\mathcal{Y}} &  \tilde{\triangle}_1 \cap  \tilde{\nabla}_0 \\
\tilde{\triangle}_1 \cap  \tilde{\nabla}_0 &\xrightarrow{\mathcal{Y}} &  \tilde{\triangle}_0 \cap  \tilde{\nabla}_1 \\
\tilde{\triangle}_1 \cap  \tilde{\nabla}_1 &\xrightarrow{\mathcal{Y}} &  \tilde{\triangle}_1 \cap  \tilde{\nabla}_1 \\
\tilde{\triangle}_D &\xrightarrow{\mathcal{Y}} &  \tilde{\nabla}_D \\
\end{eqnarray*}
\end{prop}

\begin{proof}

We set 
\begin{eqnarray*}
\mathcal{Y} \left((\overline{\lambda}) \times [ \overline{k}] \right) &=& \mathcal{Y} \left( (\lambda_0 , \ldots, \lambda_m) \times [ k_0, \ldots, k_m] \right)  \\
&=& (\mu_0, \ldots \mu_m) \times [l_0, \ldots, l_m] \\
&=& (\overline{\mu}) \times [ \overline{l}]
\end{eqnarray*}

Then we have 
\begin{eqnarray*}
\mu_0 &=& k_0 + \ldots + k_m \\
\mu_1 &=& k_0 + \ldots + k_{m-1} \\
\mu_m &=& k_0  \\
l_0 &=&\lambda_m \\
l_m&=& \lambda_0 - \lambda_1
\end{eqnarray*}

We do each in turn.

Start with  $$(\overline{\lambda}) \times [ \overline{k}] \in  \tilde{\triangle}_0 \cap  \tilde{\nabla}_0 .$$
Hence

$$\lambda_0 < \lambda_1 + \lambda_m, k_0>k_m.$$
Then we indeed get 
$$\mu_0 < \mu_1 + \mu_m, l_0>l_m.$$

Next, suppose   $$(\overline{\lambda}) \times [ \overline{k}] \in  \tilde{\triangle}_0 \cap  \tilde{\nabla}_0 .$$

Hence

$$\lambda_0 < \lambda_1 + \lambda_m, k_0<k_m.$$

Now we have 
\begin{eqnarray*}
\mu_0 &=& k_0 + \ldots + k_{m-1} + k_m \\
&>&k_0 + \ldots + k_{m-1} + k_0 \\
&=& \mu_1 + \mu_m
\end{eqnarray*}

and

\begin{eqnarray*}
l_0 &=& \lambda_m \\
&>&\lambda_0 - \lambda_1 \\
&=&l_m
\end{eqnarray*}

The other two are similar.
\end{proof}

This gives us four relevant regions of the natural extension

$$\begin{array}{cc}
  
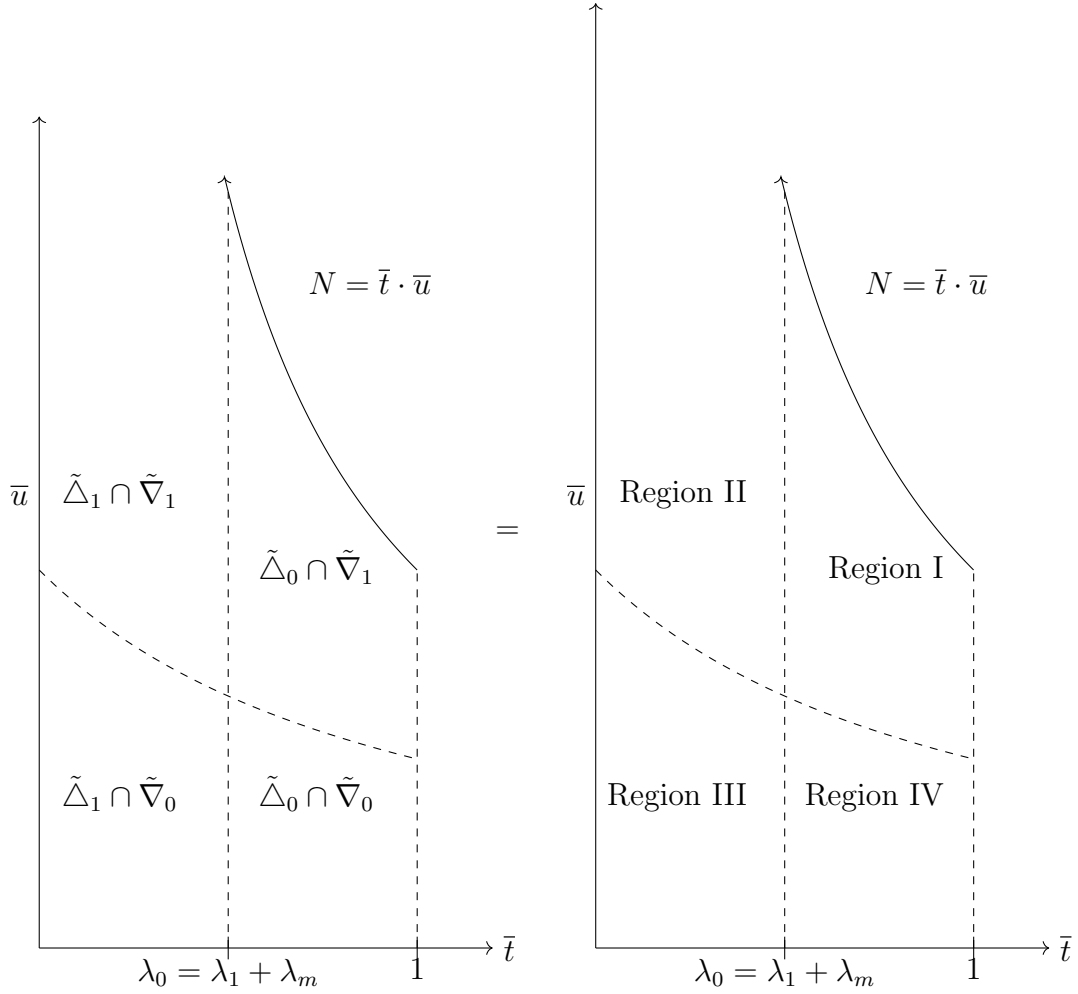
\begin{tikzpicture}[scale=5]

\draw[->](0,0)--(1.2, 0);
\draw[->](0,0)--(0, 2.2);
\draw[dashed](1,0)--(1, 1);
\draw[dashed](.5,0)--(.5, 2);

\node[right] at (1.2,0){$\overline{t}$};
\node[left]at (0,1.2){$\overline{u}$};

\draw[-](1,.03)--(1, -.03);
\node[below] at (1,0){$1$};

\draw[-](.5,.03)--(.5, -.03);
\node[below] at (.5,0){$\lambda_0= \lambda_1+\lambda_m$};

\node[left] at (1.1,1.75){$\begin{array}{c} N=\overline{t} \cdot \overline{u} \end{array}$};

\draw[scale=1, domain=.49:1, smooth, variable=\x] plot ({\x}, {1/(  \x)});

\draw[scale=1, domain=0:1, smooth, variable=\x, dashed] plot ({\x}, {1/( 1+  \x)});

\draw[->](.49,1/.49)--(.49, 1/.49);

\node[left] at (.43,1.2){$\begin{array}{c} \tilde{\triangle}_1\cap \tilde{\nabla}_1 \end{array}$};
\node[left] at (.43,.4){$\begin{array}{c} \tilde{\triangle}_1\cap \tilde{\nabla}_0 \end{array}$};

\node[left] at (.95,1){$\begin{array}{c} \tilde{\triangle}_0\cap \tilde{\nabla}_1 \end{array}$};
\node[left] at (.95,.4){$\begin{array}{c} \tilde{\triangle}_0\cap \tilde{\nabla}_0 \end{array}$};

\node[left] at (1.3,1.1){$=$};

\end{tikzpicture}

& \begin{tikzpicture}[scale=5]

\draw[->](0,0)--(1.2, 0);
\draw[->](0,0)--(0, 2.5);
\draw[dashed](1,0)--(1, 1);
\draw[dashed](.5,0)--(.5, 2);

\node[right] at (1.2,0){$\overline{t}$};
\node[left]at (0,1.2){$\overline{u}$};

\draw[-](1,.03)--(1, -.03);
\node[below] at (1,0){$1$};

\draw[-](.5,.03)--(.5, -.03);
\node[below] at (.5,0){$\lambda_0= \lambda_1+\lambda_m$};

\node[left] at (1.1,1.75){$\begin{array}{c} N=\overline{t} \cdot \overline{u} \end{array}$};

\draw[scale=1, domain=.49:1, smooth, variable=\x] plot ({\x}, {1/(  \x)});

\draw[scale=1, domain=0:1, smooth, variable=\x, dashed] plot ({\x}, {1/( 1+  \x)});

\draw[->](.49,1/.49)--(.49, 1/.49);

\node[left] at (.43,1.2){Region II};
\node[left] at (.43,.4){Region III};

\node[left] at (.95,1){Region I};
\node[left] at (.95,.4){Region IV};

\end{tikzpicture}
\end{array} $$

Proposition \ref{internal symmetry} is simply saying that 
\begin{eqnarray*}
\mbox{Region I } &\xrightarrow{\mathcal{Y}} & \mbox{Region III }  \\
\mbox{Region II }  &\xrightarrow{\mathcal{Y}} &  \mbox{Region II}  \\
\mbox{Region III } &\xrightarrow{\mathcal{Y}} & \mbox{Region I }  \\
\mbox{Region IV } &\xrightarrow{\mathcal{Y}} & \mbox{Region IV } \\
\end{eqnarray*}

The Young conjugate map is an involution about the fixed set $\mathcal{F}.$  This can be viewed diagrammatically via

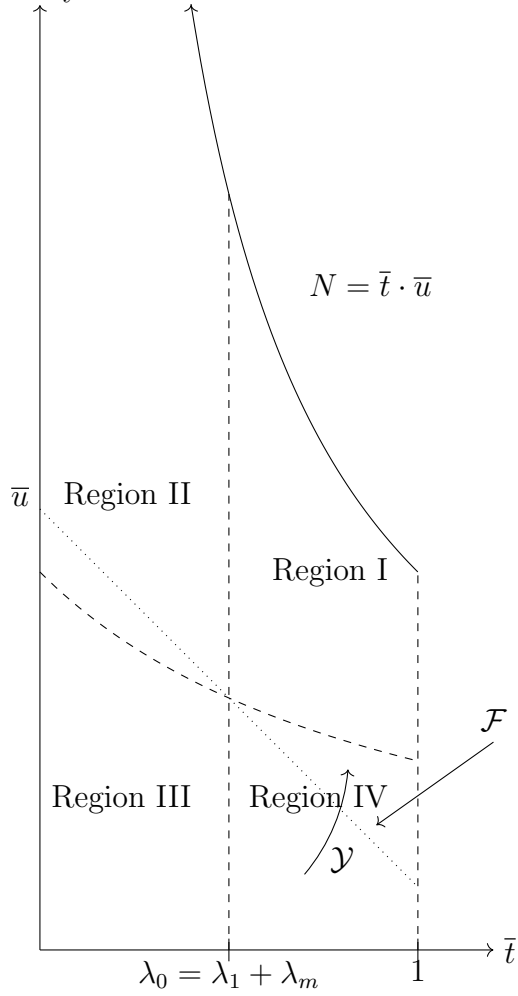
\begin{tikzpicture}[scale=5]

\draw[->](0,0)--(1.2, 0);
\draw[->](0,0)--(0, 2.5);
\draw[dashed](1,0)--(1, 1);
\draw[dashed](.5,0)--(.5, 2);

\node[right] at (1.2,0){$\overline{t}$};
\node[left]at (0,1.2){$\overline{u}$};

\draw[-](1,.03)--(1, -.03);
\node[below] at (1,0){$1$};

\draw[-](.5,.03)--(.5, -.03);
\node[below] at (.5,0){$\lambda_0= \lambda_1+\lambda_m$};

\node[left] at (1.1,1.75){$\begin{array}{c} N=\overline{t} \cdot \overline{u} \end{array}$};

\draw[scale=1, domain=.4:1, smooth, variable=\x] plot ({\x}, {1/(  \x)});

\draw[scale=1, domain=0:1, smooth, variable=\x, dashed] plot ({\x}, {1/( 1+  \x)});

\draw[scale=1, domain=0:1, smooth, variable=\x, dotted] plot ({\x},  { (-1)*\x + (7/6)});

\draw[->](2/5,5/2)--(2/5, 5/2);

\node[left] at (.43,1.2){Region II};
\node[left] at (.43,.4){Region III};

\node[left] at (.95,1){Region I};
\node[left] at (.95,.4){Region IV};

\draw[->](1.2,.55)--(.89, .33);

\node[above] at (1.2,.55){$\mathcal{F}$};

\draw[->] (.7,.2) arc (-40:-5:.5);

\node[below] at (.8,.3){$\mathcal{Y}$};

\end{tikzpicture}

Finally, note that in terms of the cylinders in Section \ref{homog cylinders}, we have 
\begin{eqnarray*}
\mbox{Region I} &=& \tilde{\triangle}_{(0) \times [1] } \\
\mbox{Region II} &=& \tilde{\triangle}_{(1) \times [1] } \\
\mbox{Region III} &=& \tilde{\triangle}_{(1) \times [0] } \\
\mbox{Region IV} &=& \tilde{\triangle}_{(0) \times [0] }
\end{eqnarray*}
We believe this is new even in the extremely classical $m=1$ case.

\subsection{The non-homogeneous  or affine case}

We want to find the non-homogeneous  or affine Young conjugate map $\mathcal{Y}_{\mbox{A}}$.

For a given 
$$(t_1, \ldots, t_m) \times [u_1, \ldots, u_m]$$
set
$$u_0 = N - t_1 u_1 - \ldots - t_m u_m.$$

The affine Young conjugate map is the composition of the following (for any fixed positive number $N$):

$$\begin{array}{c}
(t_1, \ldots, t_m) \times  [u_1, \ldots, u_m)] \\ \\
\downarrow \\ \\
\left(  1  ,    t_1  ,   \ldots,    t_m      \right)  \times [u_0, u_1, \ldots, u_m]  \\ \\
\downarrow \mathcal{Y} \\ \\
(u_0 + \ldots + u_m, u_0 + \ldots + u_{m-1}, \ldots, u_0) \times [t_m, t_{m-1}-t_m, t_{m-2} - t_{m-1} , \ldots, 1 - t_1] \\ \\
\downarrow \\ \\
\left(   \frac{  u_0 + \ldots + u_{m-1}}{ u_0 + \ldots + u_{m}}, \ldots,  \frac{  u_0 }{ u_0 + \ldots + u_{m}} \right) \\
\times \\
 \left[ (u_0 + \ldots + u_{m})(t_{m-1}- t_m) , \ldots , (u_0 + \ldots + u_{m})(1- t_1)\right] \\
  \end{array}$$

For $m=1$, we have then 
$$\begin{array}{c}
(t) \times  [u] \\ \\
\downarrow  \mathcal{Y}_{\mbox{A}}   \\ \\
\\
\left(   \frac{  N-t u}{ N-tu + u}\right)   \times 
 \left[ (N-tu + u) (1- t)\right] \\
  \end{array}$$

Again we believe this is new even in the  $m=1$ case.

\section{The homogeneous fast extended triangle map }\label{fast homogeneous case}

\subsection{The maps}

Given any two maps $\tilde{T}_0$ and $\tilde{T}_1$, it is standard to collate the various $  \tilde{T  } _1$'s, looking at the map
$$ \tilde{G  }_n =  \tilde{T  }_0 \circ  \tilde{T  }_1^{n}.$$
The $ \tilde{G  }_n$ are called the fast version, the multiplicative version or the Gauss version.
We want to find the natural extensions for the Gauss maps associated to the triangle map.  We will also want to find an associated map $  \tilde{G  }_n^*$ so that
$$\mathcal{Y} =  \tilde{G  }_n^* \circ \mathcal{Y} \circ  \tilde{G  }_n.$$

Via direct calculations we get that 

$$\begin{array}{c}
(\lambda_0, \ldots, \lambda_m) \times [k_0, k_1, \ldots, k_m] \\ \\
\downarrow \tilde{G}_n =  \tilde{T  }_0 \circ  \tilde{T  }_1^{(n)}  \\
\\
(\lambda_1, \lambda_2, \ldots, \lambda_m, \lambda_0-\lambda_1 - n\lambda_m) \times [k_0+k_1, k_2, k_3, \ldots, k_{m-1}, nk_0 +k_m, k_0]
  \end{array}$$

 From  \cite{BBDGI}, we have that 
  $$ \tilde{G  }_n^* =  \tilde{T  }_1^{(n)} \circ  \tilde{T  }_0,$$
  giving us

$$\begin{array}{c}
(\lambda_0, \ldots, \lambda_m) \times [k_0, k_1, \ldots, k_m] \\ \\
\downarrow \tilde{G}_n^* =    \tilde{T  }_1^{(n)} \circ  \tilde{T  }_0 \\ 
\\
(\lambda_1- n(\lambda_0-\lambda_m) , \lambda_2, \ldots, \lambda_m, \lambda_0-\lambda_1) \times [k_0+k_1, k_2, k_3, \ldots, k_{m}, (n+1) k_0 + n k_1]
  \end{array}$$
  
  We have for $m>1$ 
$$
\mbox{Domain}(\tilde{G}_N)  = \{ \lambda_0 - \lambda_1 - N \lambda_m >0 >  \lambda_0 - \lambda_1 -(N+1)  \lambda_m , k_0>k_m  \}  $$
(the $k_0>k_m$ term is since we must be in the image of $\tilde{T}_0$)
and 
$$\mbox{Domain}(\tilde{G}_N^*)  = \left\{  \begin{matrix} \lambda_0 < \lambda_1+ \lambda_m \\
(N+1) \lambda_1 - N\lambda_0 > \lambda_2 > (N+2) \lambda_1 - (N+1)\lambda_0 \\
\end{matrix} \right\}
$$
( The $\lambda_0 < \lambda_1+ \lambda_m$ is since we must be in the domain of $\tilde{T}_0$)
 and for $m=1$

$$ \mbox{Domain}(\tilde{G}_N)  = \{ \lambda_0 - (N+1) \lambda_1 >0 >  \lambda_0 - \lambda_1 -(N+2)  \lambda_m, k_0>k_1   \}  $$
$$\mbox{Domain}(\tilde{G}_N^*)  =   \left\{  \begin{matrix} \lambda_0 < 2\lambda_1 \\
 (N+2) \lambda_1 > (N+1)\lambda_0 \\
(N+3) \lambda_1 < (N+2) \lambda_0\end{matrix} \right\}
 $$
 The co-domains are, 
 for $m>1$,

$$\mbox{Image}(\tilde{G}_N)  = \{ k_m < k_0, k_{m-1} - Nk_m >0 >  k_{m-1} - (N+1)k_m \} $$
and
$$
\mbox{Image}(\tilde{G}_N^*)  = \left\{ \lambda_0 < \lambda_1 + \lambda_m, Nk_0 < k_m < (N+1) k_0\right\} 
 $$
and for $m=1$
 $$
\mbox{Image}(\tilde{G}_N)  = \{ k_0- (N+1)k_1>0 > k_0- (N+2)k_1  \}  $$
$$\mbox{Image}(\tilde{G}_N^*)  = \left\{\lambda_0 < 2 \lambda_1,  Nk_0 < k_m < (N+1) k_0   \right\}
 $$

We have 
\begin{prop} Young conjugation $\mathcal{Y}$ is a one-to-one onto map
$$\mathcal{Y}:    \mbox{Domain}(\tilde{G}_N) \rightarrow   \mbox{Image}(\tilde{G}_N^*) $$
and
$$\mathcal{Y}:    \mbox{Image}(\tilde{G}_N) \rightarrow   \mbox{Domain}(\tilde{G}_N^*) $$
\end{prop}

\begin{proof}These are calculations.
\end{proof}

\begin{definition}
The Gauss natural extension of  the triangle map 
$$\tilde{G}: \tilde{\triangle} \rightarrow \tilde{\triangle} $$
is defined  by setting
$$\tilde{G} (\overline{\lambda} )\times [ \overline{k} ] = \tilde{G}_N (\overline{\lambda} )\times [ \overline{k} ] $$
when
$$(\overline{\lambda} )\times [ \overline{k} ] \in  \mbox{Domain}(\tilde{G}_N) .$$
\end{definition}

\subsection{The cylinders}\label{fast homog cylinders}
As to be expected, this is almost exactly the same as in Subsection \ref{homog cylinders}.

Given two sequences $(i_0, \ldots, i_{n_1})$ and $[j_0, \ldots, j_{n_2}]$, each consisting of positive integers,  define the corresponding  cylinder

$$\tilde{\triangle}_{ (i_0, \ldots, i_{n_1})\times [j_0, \ldots, j_{n_2}]   }^G  $$
to be all $\overline{\lambda} \times \overline{k} \in \tilde{\triangle}$ such that 
$$   \tilde{G}^k(\overline{\lambda} \times \overline{k}) \in  \mbox{Domain}(\tilde{G}_{i_k})   \tilde{G}^{-l}(\overline{\lambda} \times \overline{k}) \in   \mbox{Image}(\tilde{G}_{i_l}),$$
for any  integers $k$ and $l$ with $0\leq k \leq n_1$ and any $0\leq l \leq n_2.$

We have 

\begin{prop} For all integers  $k$ with $0\leq k \leq n_1$, we have 
$$\tilde{G}^{k}\tilde{\triangle}_{  (i_0, \ldots, i_{n_1})\times (j_0, \ldots, j_{n_2})   }^G =   \tilde{\triangle}_{ (i_{k+1},  \ldots, i_{n_1}) \times  (i_k, i_{k-1}, \ldots, i_{k_0}, j_0, \ldots, j_{n_2} ) }^G  $$
\end{prop}

 As with Theorem \ref{homog natural exension} , we have 
 \begin{theorem}   The map $\tilde{G}$ is the natural extension of the map $T$.

\end{theorem}
 As before,  proof is the same as the earlier theorem.

\section{The non-homogeneous fast extended triangle map or the fast affine triangle map}\label{nonhomogeneous fast case}

\subsection{The maps}

Our domain is 
$$ \tilde{\triangle}_m^A(N) = \{ (t_1, \ldots, t_m)\times [u_1, \ldots, u_m]\in \R^m\times \R^m: t_1 > \cdots > t_m>0, u_i>0,  N- \sum_{i=1}^m t_i u_i >0\}.$$
The map from the affine triangle to the homogeneous domain is 
$$\begin{array}{c}
(t_1, \ldots, t_m) \times [u_1, \ldots, u_m]\\
\downarrow  \\

(1, t_1, \ldots, t_m) \times [u_0, u_1, \ldots, u_m] \\
  \end{array}$$
where $u_0 = N- ( t_1 u_1 + \ldots + t_m u_m).
$

The affine or non-homogeneous fast extended triangle map $\tilde{G}_n^A $ will be the composition of

$$\begin{array}{c}
(t_1, \ldots, t_m) \times [u_1, \ldots, u_m]\\ \\
\downarrow  \\ \\

(1, t_1, \ldots, t_m) \times [u_0, u_1, \ldots, u_m] \\ \\
\downarrow \tilde{G}_n =  \tilde{T  }_0 \circ  \tilde{T  }_1^{(n)}  \\ \\

(t_1, t_2, \ldots, t_m, 1-t_1 - nt_m) \times [u_0+u_1, u_2, u_3, \ldots, u_{m-1}, nu_0 +u_m, u_0] \\ \\
\downarrow  \pi \\ \\
\left(  \frac{t_2}{t_1}, \ldots, \frac{t_m}{t_1}, \frac{1-t_1 - n t_m}{t_1} \right) \times [t_1u_2, \ldots, t_1 u_{m-1}, t_1(nu_0 + u_m), t_1 u_0]
  \end{array}$$

The associated affine version of the map $ \tilde{G}_n^* $ (which we will denote in the admittedly awkward fashion as $(\tilde{G}_n^*)^A$) will be the composition of 

$$\begin{array}{c}
(t_1, \ldots, t_m) \times [u_1, \ldots, u_m]\\
\downarrow  \\
(1, t_1, \ldots, t_m) \times [u_0, u_1, \ldots, u_m] \\ \\
\downarrow (\tilde{G}_n^*)^A =    \tilde{T  }_1^{(n)} \circ  \tilde{T  }_0 \\ \\
(t_1- n(1-t_m) , t_2, \ldots, t_m, 1-t_m) \times [u_0+u_1, u_2, u_3, \ldots, u_{m}, (n+1) u_0 + n u_1] \\ \\
\downarrow \pi  \\ \\
\left( \frac{t_2}{t_1-n(1-t_m)} , \ldots,  \frac{t_m}{t_1-n(1-t_m)},  \frac{1-t_m}{t_1-n(1-t_m)}   \right) \\
\times  \\
\left [(t_1-n(1-t_m)u_2, \ldots, (t_1-n(1-t_m)u_m, (t_1-n(1-t_m)((n+1)u_0 + n u_1)\right]

  \end{array}$$
  
  For the $m=1$ case, we have 
  affine or non-Homogeneous fast extended triangle map $\tilde{G}_n^A $ will be
\begin{eqnarray*}
& (t) \times [u]    &   \\ \\
&  \downarrow \tilde{G}_n^A &   \\ \\
&    \left(  \frac{1-(n+1) t}{t} \right) \times [t(N-tu)] & 
\end{eqnarray*}
and the  map  $(\tilde{G}_n^*)^A $ will be 
\begin{eqnarray*}
& (t) \times [u]    &    \\
&    \downarrow (\tilde{G}_n^*)^A & \\
&     \left(  \frac{1-  t}{(n+1) t - n} \right) \times [((N-tu)+ u))(n+1)(N-tu) + nu)]   &
\end{eqnarray*}

 We have for $m>1$, 
 $$\mbox{Domain}(\tilde{G}_n^A)  = \{ 1 - t_1 - n t_m >0 > 1 - t_1 -(n+1)  t_m , u_0>u_m  \} $$
and
$$\mbox{Domain}((\tilde{G}_n^*)^A)  = \left\{  \begin{matrix} 1 < t_1+ t_m \\
(n+1) t_1 - n > t_2 > (n+2)t_1  - (n+1) \\
\end{matrix}\right\} $$

 For $m=1$,
 $$\mbox{Domain}(\tilde{G}_n^A)  = \{ 1- (n+1) t >0 >  1 -(n+2) t, N- tu > u   \} $$
 and 
 $$\mbox{Domain}((\tilde{G}_n^*)^A) =   \left\{  \begin{matrix} 1 < 2\ t \\
 (n+2)t > (n+1) \\
(n+3) t < (n+2)\end{matrix} \right\}$$

 The co-domains, in turn,  are 
 for $m>1$, 
 $$\mbox{Image}(\tilde{G}_n^A)  = \{ u_m < u_0, u_{m-1} - n u_m >0 >  u_{m-1} - (n+1)   u_m \} $$
 and 
 $$\mbox{Image}((\tilde{G}_n^*)^A)  = \left\{1<t_1+t_m,  n u_0 < u_m < (n+1) u_0\right\},$$
 and for $m=1$
 
 $$\mbox{Image}(\tilde{G}_n^A)  = \{ u_0- (n+1)u>0 > u_0- (n+2)  u  \} $$
 and 
 $$\mbox{Image}((\tilde{G}_n^*)^A)  = \left\{  1<2t, n  (N-tu) < u < (n+1) (N-tu)  \right\}.$$
 
 Note on notation:  here the $n$ is for the number of iterations the $\tilde{T}_1$ term while the capital $N$ is here for the ``partition number.''

\begin{definition}
The non-homogeneous Gauss natural extension of  the triangle map 
$$\tilde{G}^A: \tilde{\triangle} \rightarrow \tilde{\triangle} $$
is defined  by setting
$$\tilde{G}^A (\overline{\lambda} )\times [ \overline{k} ] = \tilde{G}_N^A (\overline{\lambda} )\times [ \overline{k} ] $$
when
$$(\overline{\lambda} )\times [ \overline{k} ] \in  \mbox{Domain}(\tilde{G}_N^A) .$$
\end{definition}

 \subsection{The cylinders}
 This is the same as in Subsection \ref{fast homog cylinders}.  As far as notation goes, we are simply adding a subscript of $A$ for affine at appropriate places.

Given two sequences $(i_0, \ldots, i_{n_1})$ and $[j_0, \ldots, j_{n_2}]$, each consisting of positive integers,  define the corresponding  cylinder

$$\tilde{\triangle}_{ (i_0, \ldots, i_{n_1})\times [j_0, \ldots, j_{n_2}]   }^{G,A}  $$
to be all $\overline{\lambda} \times \overline{k} \in \tilde{\triangle}$ such that 
$$  ( \tilde{G^A})^k(\overline{\lambda} \times \overline{k}) \in  \mbox{Domain}(\tilde{G}_{i_k})^A,   (\tilde{G}^A)^{-l}(\overline{\lambda} \times \overline{k}) \in   \mbox{Image}(\tilde{G}_{i_l})^A,$$
for any  integers $k$ and $l$ with $0\leq k \leq n_1$ and any $0\leq l \leq n_2.$

We have 

\begin{prop} For all integers  $k$ with $0\leq k \leq n_1$, we have 
$$  (\tilde{G}^A)^{k}\tilde{\triangle}_{  (i_0, \ldots, i_{n_1})\times (j_0, \ldots, j_{n_2})   }^{G,A} =   \tilde{\triangle}_{ (i_{k+1},  \ldots, i_{n_1}) \times  (i_k, i_{k-1}, \ldots, i_{k_0}, j_0, \ldots, j_{n_2} ) }^{G,A}  $$
\end{prop}

 Just like the earlier three maps, we have 
 an analog of  Theorem \ref{homog natural exension}:
 \begin{theorem}   The map $\tilde{G}$ is the natural extension of the map $T$.

\end{theorem}
 And as before,  proof is the same.

\section{Graphs of the fast natural extensions}\label{Graphs of the fast natural extensions}
 When $m = 1$ (and setting $N=1$) , we have that  the domains of  the $\tilde{G}_n^A$ create natural refinement of the diagram in Section 
  \ref{m1 case}:

\begin{center}
\begin{tikzpicture}[scale=6]

\draw[->](0,0)--(1.2, 0);
\draw[->](0,0)--(0, 1.2);
\draw[dashed](1,0)--(1, 1/2);

\node[right] at (1.2,0){$t$};
\node[left]at (0,1.2){$u$};

\draw[-](1,.03)--(1, -.03);
\node[below] at (1,-.05){$1$};

\draw[-](.03, 1)--(-.03, 1);
\node[left] at (0,1){$1$};

\draw[-](.03, 1/2)--(-.03, 1/2);
\node[left] at (0,1/2){$1/2$};

\draw[scale=1, domain=0:1, smooth, variable=\t] plot ({\t}, {1/(1+ \t)});

\draw[dashed](1/2,0)--(1/2, 2/3);
\draw[dashed](1/3,0)--(1/3, 3/4);
\draw[dashed](1/4,0)--(1/4, 4/5);
\draw[dashed](1/5,0)--(1/5, 5/6);

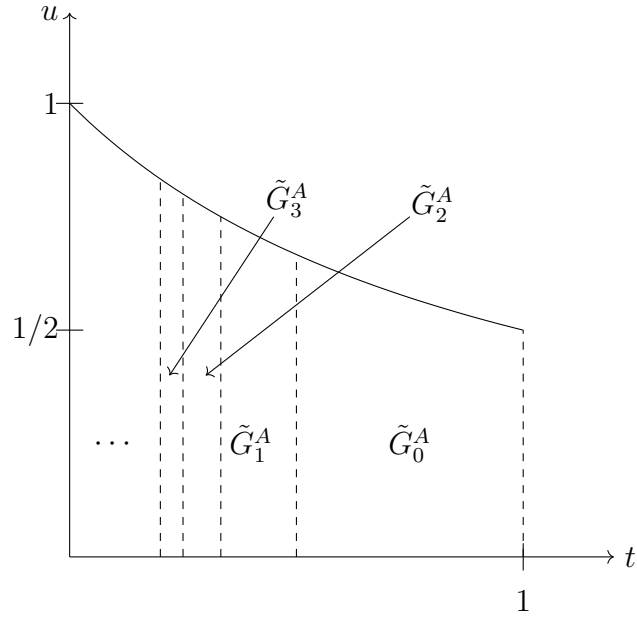
\node[] at (.1,.25){$\cdots$};

\node[] at (.75,.25){$\tilde{G}_0^A$};
\node[] at (.4,.25){$\tilde{G}_1^A$};

\draw[->](3/4,3/4)--(.3, .4);
\node[] at (.8,.78){$\tilde{G}_2^A$};

\draw[->](.45,3/4)--(.22, .4);
\node[] at (.48,.79){$\tilde{G}_3^A$};

\end{tikzpicture}
\end{center}
Here the $\tilde{G}_n^A$s refer to their domains.

Similarly the co-domains create a natural refinement of the diagram in Section 
  \ref{m1 case}:

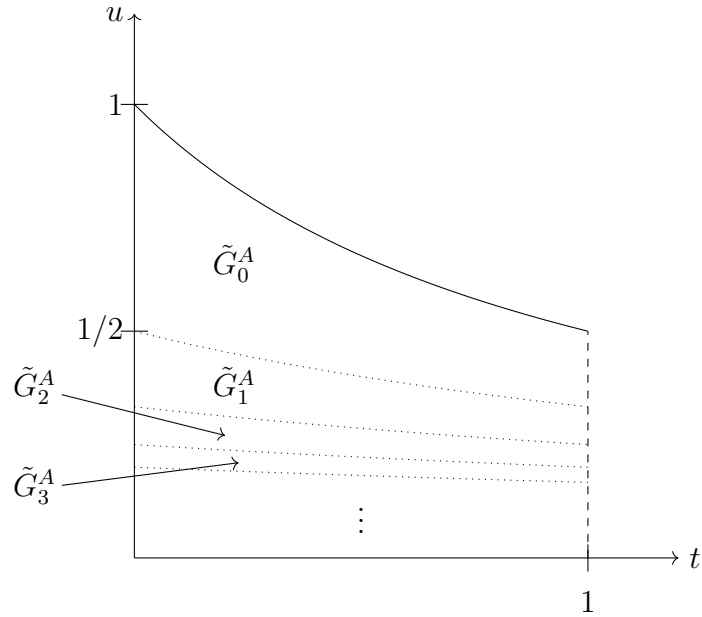
\begin{center}
\begin{tikzpicture}[scale=6]

\draw[->](0,0)--(1.2, 0);
\draw[->](0,0)--(0, 1.2);
\draw[dashed](1,0)--(1, 1/2);

\node[right] at (1.2,0){$t$};
\node[left]at (0,1.2){$u$};

\draw[-](1,.03)--(1, -.03);
\node[below] at (1,-.05){$1$};

\draw[-](.03, 1)--(-.03, 1);
\node[left] at (0,1){$1$};

\draw[-](.03, 1/2)--(-.03, 1/2);
\node[left] at (0,1/2){$1/2$};

\draw[scale=1, domain=0:1, smooth, variable=\t] plot ({\t}, {1/(1+ \t)});

\draw[dotted, scale=1, domain=0:1, smooth, variable=\t] plot ({\t}, {1/(2+ \t)});
\draw[dotted, scale=1, domain=0:1, smooth, variable=\t] plot ({\t}, {1/(3+ \t)});
\draw[dotted, scale=1, domain=0:1, smooth, variable=\t] plot ({\t}, {1/(4+ \t)});
\draw[dotted, scale=1, domain=0:1, smooth, variable=\t] plot ({\t}, {1/(5+ \t)});

\node[] at (1/2,1/10){$\vdots$};

\node[] at (.22,.65){$\tilde{G}_0^A$};
\node[] at (.22,.38){$\tilde{G}_1^A$};

\draw[->](-.16, .36)--(.2, .27);
\node[] at (-.22,.38){$\tilde{G}_2^A$};

\draw[->](-.16, .16)--(.23, .21);
\node[] at (-.22,.16){$\tilde{G}_3^A$};

\end{tikzpicture}

\end{center}
Here the $\tilde{G}_n^A$s refer to their co-domains.

 And when $m = 1$ (and setting $N=1$) , we have that  the domains of  the $(\tilde{G}_n^*)^A$   also create natural refinement of the diagram in Section 
  \ref{m1 case}:

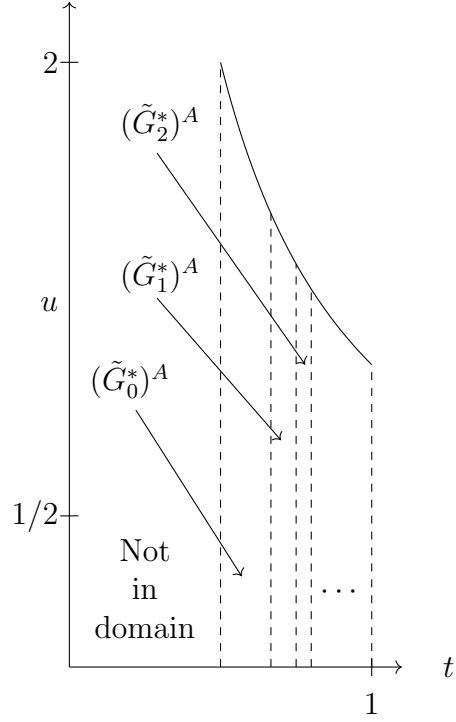
\begin{center}
\begin{tikzpicture}[scale=4]

\draw[->](0,0)--(1.1, 0);
\draw[->](0,0)--(0, 2.2);
\draw[dashed](1,0)--(1, 1);

\node[right] at (1.2,0){$t$};
\node[left]at (0,1.2){$u$};

\draw[-](1,.03)--(1, -.03);
\node[below] at (1,-.05){$1$};

\draw[-](.03, 2)--(-.03, 2);
\node[left] at (0,2){$2$};

\draw[-](.03, 1/2)--(-.03, 1/2);
\node[left] at (0,1/2){$1/2$};

\draw[scale=1, domain=1/2:1, smooth, variable=\t] plot ({\t}, {1/( \t)});

\draw[dashed](1/2,0)--(1/2, 2);
\draw[dashed](2/3,0)--(2/3, 3/2);
\draw[dashed](3/4,0)--(3/4, 4/3);
\draw[dashed](4/5,0)--(4/5, 5/4);

\node[] at (.2,.95){$  (\tilde{G}_0^*)^A  $};
\draw[->](.22, .85)--(.57, .3);

\node[] at (.3,1.3){$  (\tilde{G}_1^*)^A  $};
\draw[->](.29, 1.22)--(.7, .75);

\node[] at (.3,1.8){$  (\tilde{G}_2^*)^A  $};
\draw[->](.29, 1.7)--(.78, 1);

\node[]at (.9, .25){$\cdots$};

\node[]at (.25, .25){$\begin{array}{c}  \mbox{Not}  \\ \mbox{in} \\ \mbox{domain} \end{array}$};

\end{tikzpicture}
\end{center}
Here the $(\tilde{G}_n^*)^A$ s refer to their domains.

And the co-domains of  the $(\tilde{G}_n^*)^A$   create natural refinement of the diagram in Section 
  \ref{m1 case}:

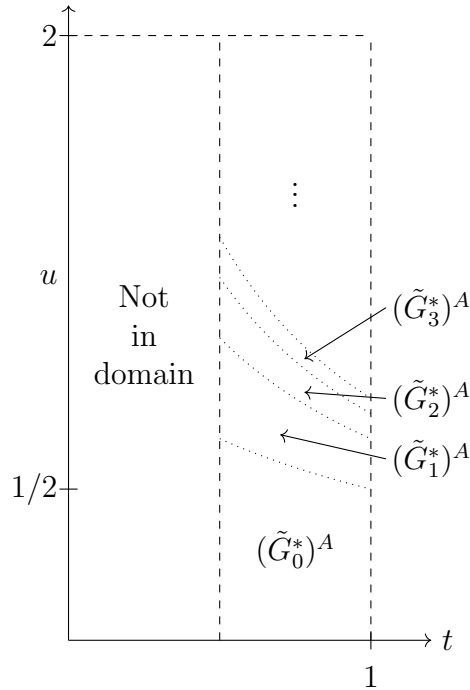
\begin{center}
\begin{tikzpicture}[scale=4]

\draw[->](0,0)--(1.2, 0);
\draw[->](0,0)--(0, 2.1);
\draw[dashed](1,0)--(1, 2);

\node[right] at (1.2,0){$t$};
\node[left]at (0,1.2){$u$};

\draw[-](1,.03)--(1, -.03);
\node[below] at (1,-.05){$1$};

\draw[-](.03, 2)--(-.03, 2);
\node[left] at (0,2){$2$};

\draw[-](.03, 1/2)--(-.03, 1/2);
\node[left] at (0,1/2){$1/2$};

\draw[dotted, scale=1, domain=1/2:1, smooth, variable=\t] plot ({\t}, {1/(1+ \t)});
\draw[dotted, scale=1, domain=1/2:1, smooth, variable=\t] plot ({\t}, {2/(1+2* \t)});
\draw[dotted, scale=1, domain=1/2:1, smooth, variable=\t] plot ({\t}, {3/(1+ 3*\t)});
\draw[dotted, scale=1, domain=1/2:1, smooth, variable=\t] plot ({\t}, {4/(1+4* \t)});

\draw[dashed](1/2,0)--(1/2, 2);
\draw[dashed](0,2)--(1, 2);

\node[]at (.25, 1){$\begin{array}{c}  \mbox{Not}  \\ \mbox{in} \\ \mbox{domain} \end{array}$};

\node[] at (.75,.3){$  (\tilde{G}_0^*)^A  $};

\node[] at (1.2,.6){$  (\tilde{G}_1^*)^A  $};
\draw[->](1.05, .6)--(.7, .68);

\node[] at (1.2,.8){$  (\tilde{G}_2^*)^A  $};
\draw[->](1.05, .8)--(.78, .82);

\node[] at (1.2,1.1){$  (\tilde{G}_3^*)^A  $};
\draw[->](1.05, 1.1)--(.78, .93);

\node[] at (.75,1.5){$ \vdots $};

\end{tikzpicture}
\end{center}
Here the $(\tilde{G}_n^*)^A$ s refer to their co-domains.

We have 
\begin{prop} Young conjugation $\mathcal{Y}$ is a one-to-one onto map
$$\mathcal{Y}:    \mbox{Domain}(\tilde{G}_N^A) \rightarrow   \mbox{Image}(\tilde{G}_N^*)^A $$
and
$$\mathcal{Y}:    \mbox{Image}(\tilde{G}_N)^A \rightarrow   \mbox{Domain}(\tilde{G}_N^*)^A$$
\end{prop}

\begin{proof}These are calculations.
\end{proof}

\section{The invariant measure} \label{The invariant measure} 

\subsection{The relevant Jacobians}
 Consider
the simplex 
$$\triangle_m^A= \{ (t_1, \ldots , t_m) \in \R^m:  1>t_1 > \cdots > t_m >0\}.$$
We have the two maps
$$T:\triangle_m^A \rightarrow \triangle_m^A$$
and
$$G:\triangle_m^A \rightarrow \triangle_m^A$$
defined by 

$$T(t_1, \ldots, t_m) = \left\{  \begin{array}{ccc} \left( \frac{t_2}{t_1} , \ldots, \frac{t_m}{t_1}, \frac{1-t_1}{t_1} \right)& \mbox{if}& 1< t_1 + t_m  \\
 \left( \frac{t_1}{ 1 - t_m} , \ldots, \frac{t_m}{1- t_m} \right)& \mbox{if}& 1> t_1 + t_m \end{array} \right.$$
 and 
 $$G(t_1, \ldots, t_m) =   \begin{array}{ccc} \left( \frac{t_2}{t_1} , \ldots, \frac{t_m}{t_1}, \frac{1-t_1 - k t_m}{t_1} \right)& \mbox{if}& 1-t_1-kt_m \geq 0 >1-t_1-(k+1)t_m   \end{array} .$$

\begin{prop} We have for all points $\overline{t} = (t_1, \ldots, t_m) \in \triangle_m^A$ that 
$$\mbox{Jacobian}(G) (\overline{t}) = \frac{1}{t_1^{m+1}}.$$
If $ 1< t_1+ t_m$, then 
$$\mbox{Jacobian}(T) (\overline{t}) = \frac{1}{t_1^{m+1}},$$
and if $1>t_1+ t_m$, then
$$\mbox{Jacobian}(T) (\overline{t}) = \frac{1}{(1-t_m)^{m+1}},$$
\end{prop}
\begin{proof}
These are calculations.
\end{proof}
 \begin{corollary} Let $\mu$ be Lebesgue measure on the simplex $\triangle_m^A$.  Then for measurable sets $U$ in $\triangle_m^A$, it is not necessarily the case that 
 $$\mu(U) = \mu(T^{-1}(U))$$
 or that 
 $$\mu(U) = \mu(G^{-1}(U)).$$
Thus the Lebesgue measure is not an invariant measure for either  map $T$ or for map $G$.
 \end{corollary}
 
The natural question is to find an invariant measure.  This is classical for $m=1$, done in \cite{Ass} for the Gauss map, and completely worked out in \cite{Garrity-Lehmann Duke}, where it was shown that in Theorem 7.1 that the invariant measure for $G$ is 
$$\frac{\mathrm{\mu}}{t_1 t_2 \ldots t_{m-1} ( 1 + t_m)}$$
(perversely, in that paper the map $G$ was denoted by $T$) and in Theorem 8.3 the invariant measure for our $T$ is 
$$\frac{\mathrm{\mu}}{t_1 t_2 \ldots t_{m}}.$$
The method in \cite{Garrity-Lehmann Duke}  for deriving these invariant measures was the time-honored approach of guessing, then checking.

There is another approach for finding invariant measures if the natural extension is known.  This is highlighted in Arnoux and Nogueira \cite{Arnoux-Nogueira-93}.  We start with observing that 

\begin{theorem} For all of our  natural extension maps, we have
$$|\mbox{Jacobian}(\tilde{T}_0^A )| = |\mbox{Jacobian}(\tilde{T}_1^A )| =  |\mbox{Jacobian}(\tilde{G}_n^A )| =1. $$

\end{theorem}
\begin{proof} These are calculations.
\end{proof}

As we now know the natural extensions for the maps $T$ and $G$, we have a non-guessing approach for finding the corresponding invariant measures.

\subsection{The $T$ case}
(As before, we are setting $N=1$, in the formulas for the domain
$$ \tilde{\triangle}_m^A(N) = \{ (t_1, \ldots, t_m)\times [u_1, \ldots, u_m]\in \R^m\times \R^m: t_1 > \cdots > t_m>0, u_i>0,  N- \sum_{i=1}^m t_i u_i >0\}.)$$

Set $$t=(t_1, \ldots, t_m) \in \triangle_m^A.$$
We need to find a function $f(t)$ so that the measure
$$f(t)\mathrm{d} \mu = f(t) \mathrm{d} t_1 \cdots  \mathrm{d} t_m$$
is an invariant measure for the map $T$.
Fix the point $t$.  Consider the set 
$$D(t) = \{ (u_1, \ldots, u_m) \in \R^m:u_i\geq 0, 0\leq t_1u_1 + \ldots + t_m u_m \leq 1\}.$$
Then the desired function $f(t)$ is 
$$f(t) = \int_{D(t)}  \mathrm{d} u_1 \cdots  \mathrm{d} u_m.$$
As this is discussed for  the general case in \cite{Arnoux-Nogueira-93}, we will simply show that it works for the $m=3$ case here.

We have
\begin{eqnarray*}
\int_{D(t)}  \mathrm{d} u_1 \cdots  \mathrm{d} u_3 &=& \int_0^{1/t_3}   \int_0^{(1-t_3u_3 )/t_2}
     \int_0^{(1-t_2u_2 - t_3 u_3)/t_1} \mathrm{d} u_1\mathrm{d} u_2 \mathrm{d} u_3 \\
     &=& \frac{1}{t_1 }    \int_0^{1/t_3}   \int_0^{(1-t_3u_3 )/t_2}
   (1-t_2u_2 - t_3 u_3)\mathrm{d} u_2 \mathrm{d} u_3    \\
   &=&  \frac{1}{t_1 }    \int_0^{1/t_3} \left( (1-t_3u_3) \left(    \frac{1-t_3 u_3}{t_2}   \right) - \frac{t_2}{2} \left(    \frac{1-t_3 u_3}{t_2}   \right)^2\right) \mathrm{d} u_3 \\
   &=& \frac{1}{2t_1t_2}  \int_0^{1/t_3} (1-t_3u_3)^2 \mathrm{d} u_3 \\
   &=& \frac{1}{3\cdot 2\cdot  t_1 t_2 t_3}.
   \end{eqnarray*}
   As the invariant measure is only defined up to a positive constant, we are done.

   \subsection{The $G$ case}
   
   Our domain for the natural extension is now 
   $$\mbox{Domain}(\tilde{G}) = \{ t \times u:1\geq t_1 >\cdots >t_m >0, u_i>0, 1> t_1u_1 + \ldots + t_{m-1} u_{m-1} + (t_m+1) u_m\}$$
   As before, we want  to find a function $g(t)$ so that the measure
$$g(t)\mathrm{d} \mu = g(t) \mathrm{d} t_1 \cdots  \mathrm{d} t_m$$
is an invariant measure for the map $G$.

   Fix the vector $t=(t_1, \ldots, t_m) $
   Set
   $$D(t) = \{ (u_1, \ldots, u_m) : (t_1, \ldots, t_m)\times (u_1, \ldots, u_m) \in \mbox{Domain}(\tilde{G}) \}.$$
   Then 
   $$g(t) =  \int_{D(t)}  \mathrm{d} u_1 \cdots  \mathrm{d} u_m.$$
   
   We will again simply show that it works for the $m=3$ case here.

We have
\begin{eqnarray*}
\int_{D(t)}  \mathrm{d} u_1 \cdots  \mathrm{d} u_3 &=& \int_0^{1/(1+t_3)}   \int_0^{(1-(1+t_3)u_3 )/t_2}
     \int_0^{(1-t_2u_2 - (1+t_3)u_3)/t_1} \mathrm{d} u_1\mathrm{d} u_2 \mathrm{d} u_3 \\
     &=& \frac{1}{t_1 }    \int_0^{1/(1+t_3)}   \int_0^{(1-(1+t_3)u_3 )/t_2}
   (1-t_2u_2 - (1+t_3) u_3)\mathrm{d} u_2 \mathrm{d} u_3    \\
   &=& \frac{1}{2t_1t_2}  \int_0^{1/(1+t_3)} (1-(1+ t_3)u_3)^2 \mathrm{d} u_3 \\
   &=& \frac{1}{3\cdot 2\cdot  t_1 t_2 (1+t_3)}.
   \end{eqnarray*}
   As the invariant measure is only defined up to a positive constant, we are done.

\section{ On Siukaev's work}\label{Siukaev}

While preparing this paper (which is building on the first author's thesis \cite{Fox}), we have been in contact with  David Siukaev, whose ``A Graph-Based Method for Invariant Densities of Multidimensional Continued Fractions'' \cite{Siukaev} finds another interpretation of the natural extension of the triangle map.  We all agree that the rhetorics behind these two papers have quite  different feels.  Our approach seems to be the most natural if you want to use Young conjugations to find the symmetry of the natural extension space (which in turn is the most natural if you want to see the link with integer partitions).  Siukaev's choice is the most natural if you want to use Fougeron's rhetoric of win-loss inductions \cite{Fougeron20, Fougeron24}.  Seeing how these two papers interrelate strikes all three of us as promising.

\section{A.I. Acknowledgement}
We used Gemini AI to check for spelling and grammar errors.

\end{document}